\documentclass[a4paper]{article}
\usepackage{amsthm,latexsym,amsfonts,amsmath, amssymb,epsf}
\usepackage[labelformat=empty]{caption}
\usepackage{hhline}
\usepackage[pagebackref]{hyperref}

\usepackage{enumerate}
\usepackage{pdflscape}
\usepackage{longtable}
\usepackage{graphicx}
\usepackage{makeidx} \makeindex
\usepackage[titletoc]{appendix}
\usepackage{setspace}
\usepackage{cleveref}
\usepackage{lscape}

\newtheorem{thm}{Theorem}[section]
\newtheorem{lem}[thm]{Lemma}

\newtheorem{cor}[thm]{Corollary}

\newtheorem{re}[thm]{Result}

\theoremstyle{definition}
\newtheorem{defi}[thm]{Definition}

\newtheorem{remark}[thm]{Remark}

\newcommand{\Z}{\mathbb{Z}}

\newcommand{\Q}{\mathbb{Q}}

\DeclareMathOperator{\lcm}{lcm}
\DeclareMathOperator{\ord}{{\rm ord}}

\DeclareMathOperator{\Gal}{{\rm Gal}}

\author{
Ming Ming Tan\\
School of Computer Science and Engineering,\\
Nanyang Technological University, Singapore}

\begin{document}

\title{Group Invariant Weighing Matrices}

\maketitle

\begin{abstract}
We investigate the existence problem of group invariant matrices using algebraic approaches. 
We extend the usual concept of multipliers to group rings with cyclotomic integers as coefficients. 
This concept is combined with the field descent method and rational idempotents to develop new non-existence results. 
\end{abstract}

\bigskip

\noindent
{\bf 2010 Mathematics Subject Classification:  05B10, 05B20}\\[2mm]
{\bf Keywords: weighing matrices, circulant matrices, multipliers, Weil numbers, field descent}

\section{Introduction}
A \textbf{weighing matrix} $M = \boldsymbol{ W(v, n)}$ is a square matrix of order $v$ with entries $0, \pm 1$, and satisfying 
$$MM^{T} = nI,$$ for some positive integer $n$, where $I$ is the identity matrix. 
The integer $n$ is called the \textbf{weight} of the matrix.
The most important question in the study of weighing matrices is: 
For what values of $v$ and $n$, does there exist a $W(v,n)$? 
For the special case where $v = n$, the weighing matrix is an Hadamard matrix. 
Hadamard matrices have been studied intensively, see \cite{SeberryYamada} for an overview. 
For a survey of weighing matrices, see \cite{GeramitaSeberry}. 
See \cite{Ohmori,KoukouvinosSeberry,GysinSeberry,CraigenKharaghani,Craigen,StantonMullin,
	EadesHain, Mullin, SeberryWhiteman, Launey1984} for further references. 

\medskip

For our studies, we only focus on weighing matrices which are invariant under a group operation. 
Let $G$ be a group of order $v$. 
A matrix $M=(m_{f,h})_{f,h \in G}$ indexed with the elements of $G$ is said to be $\boldsymbol{G}$\textbf{-invariant} if 
$m_{fk, hk} = m_{f,h}$ for all $f, h, k \in G$. 
Notice that a group invariant matrix can be identified solely by its first row.

\medskip

If $G$ is a cyclic group, then a $G$-invariant weighing matrix is circulant.
Circulant weighing matrices have been studied much more intensively compared to other group invariant weighing matrices. 
We denote a \textbf{circulant weighing matrix} by $CW(v,n)$.

\medskip

In the present paper, we prove the non-existence of infinite families of group invariant weighing matrices. 
Our algebraic approaches mainly consist of the ``F-bound'', generalized multiplier theorems, analysis of cyclotomic integers of prescribed absolute value and rational idempotents. 
In particular, we solve $19$ open cases of circulant weighing matrices of order less than $200$. 
We append the most updated version of Strassler's table concerning the existence status of
circulant weighing matrices of order at most $200$ and weight at most $100$; see \Cref{table}. Detailed references for each case are provided as well. 
In addition, we solve some open cases in the table of group invariant weighing matrices of \cite{arasu2013group}, where the group is abelian but non-cyclic.
Summary of these results are shown in \Cref{table4}.
In this paper, all previously known results are labeled as ``Result xyz''.

\medskip

This paper generalizes some of the results of the author's thesis \cite{MingMingTanThesis} to abelian groups. Previously, these results were shown only for cyclic groups. 

\section{Preliminaries} 
Let $G$ be a finite abelian group and $\zeta_u$ be a primitive $u$-th root of unity.
The cyclotomic group ring $\Z[\zeta_u][G]$, is a ring that consists of elements of the form
$$X = \sum_{g \in G} X_g g$$ with $X_g \in \Z[\zeta_u]$. 
We call $X_g$ the cyclotomic \textbf{coefficient} of $g$ in $X$. 

\medskip

A group homomorphism $G \to H$ is always assumed to be extended to a group ring homomorphism $\Z[\zeta_u][G] \to \Z[\zeta_u][H]$ by linearity. 
Similarly, given a character $\chi$ of $G$, we define 
$\chi(X)=\sum_{g \in G} X_g \chi(g).$ 

\medskip

Let $t$ be an integer coprime to $u|G|$ and $\sigma_t \in {\rm Gal}(\Q(\zeta_u)/\Q)$ be the automorphism defined by $\zeta_u^{\sigma_t} = \zeta_u^t.$
We define 
$$X^{(t)} = \sum_{g \in G} X_g ^{\sigma_t} g^t.$$
Note that $X^{(-1)} = \sum_{g\in G} \overline{X}_gg^{-1}$, where the bar denotes complex conjugation.
An element $\pm\zeta_u^e g X$ for some $g \in G$ and $e \in \Z$ is called a \textbf{translate} of $X$.
We say that $X$, $Y \in \Z[\zeta_u]$ are \textbf{equivalent} if there are $\sigma\in \Gal(\Q(\zeta_u)/\Q)$ and a root of unity $\eta$ with $Y=\eta X^{\sigma}$. 

\medskip

Note that when $u=1$, the coefficient ring is the ring of integer. 
Given $D = \sum_{g \in G} a_g g \in \Z[G]$, we call the set of group elements with non-zero coefficients in $D$ the \textbf{support} of $D$, written as
$$\mbox{supp($D$)}= \left \{ g \in G: a_g \neq 0 \right \}.$$
For a subset $S$ of $G$, we will also use the same symbol $S$ to denote the corresponding group ring element $\sum_{g \in S} g$.

\medskip

We can identify a $G$-invariant matrix $M = (m_{f,h})$ with the group ring element $D = \sum_{g \in G} m_{1, g} g$ of $\Z[G]$. 
Note that $M^T$ is identified with $D^{(-1)}$ in $\Z[G]$ under this correspondence. 
Hence, a $G$-invariant weighing matrix of order $v$ and weight $n$ is equivalent to an element $D$ of $\Z[G]$ with coefficients $0, \pm 1$ only and satisfies
$$DD^{(-1)} = n.$$
This is the formulation we will use in the rest of our paper. 
Note that a weighing matrix $D = \sum_{g \in G} a_g g$ has weight $n =\sum_{g \in G}a_g^2$. The weight must be a square, as $n$ is also equal to $|\sum_{g \in G} a_g|^2$. 

\medskip

A weighing matrix $D \in \Z[G]$  is said to be \textbf{proper}
if $Dg \notin \Z[H]$ for all $g \in G$ and all proper subgroup $H$ of $G$.
In this paper, we consider only proper group invariant weighing matrices.
We call two weighing matrices $D, E \in \Z[G]$ \textbf{equivalent} if there are $g \in G$ and $t \in \Z$ with $\gcd(t,|G|)=1$ such that $E = \pm D^{(t)}g$.

\medskip

Consider the natural projection from $G$ to $H$. 
The image of a weighing matrix under this projection is called \textbf{integer weighing matrix}. 
An integer weighing matrix is denoted by $IW_a(v,n)$, where $a$ is a positive integer. 
It is defined the same way as weighing matrix $W(v,n)$, but with the entries 
not confined to $\{0, \pm 1\}$ but to $\{0, \pm 1, \pm 2, \ldots, \pm a\}$ instead. 
We denote an \textbf{integer circulant weighing matrix} by $ICW_a(v,n)$.
Note that we remove the subscript $a$ when $a = 1$. 

\medskip

In this paper, we let $C_v$ denote the cyclic group of order $v$. 
For a divisor $u$ of $v$, we always view $C_u$ as a subgroup of $C_v$. 
The identity element of a group is denoted by $1$. 
For a prime $p$ and integer $x$, let $x(p)$ be the $p$-free part of $x$, i.e., $x = x(p)p^a$ where $x(p)$ is coprime to $p$.
For relatively prime integers $t$ and $s$, we denote the multiplicative order of $t$ modulo $s$ by ${\rm ord}_s(t)$.
In addition, we denote the number of distinct prime divisors of an integer $u$ by $\delta(u)$ and the Euler totient function by $\varphi$.

\medskip

The following is a standard result, see \cite[VI, Lem. 3.5]{DesignTheory}. 

\begin{re}[Inversion Formula]\label{inversion}
	Let $G$ be an abelian group and $\hat{G}$ be the group of complex characters of $G$.
	Let $D =\sum_{g \in G} a_gg \in \Z[G]$. Then 
	$$a_g = \frac{1}{|G|}\sum_{\chi \in \hat{G}} \chi(Dg^{-1})$$ for all $g \in G$. 
\end{re}

The following result is a consequence of \Cref{inversion}. 

\begin{re}\label{character_GW}
	Let $G$ be an abelian group of order $v$. 
	Suppose $D \in \Z[G]$ has coefficients in $\{0, \pm 1, \ldots, \pm a \}$ for some integer $a$. 
	Then $D$ is a $G$-invariant $IW_a(v,n)$ if and only if $\chi(D) \overline{\chi(D)}=n$ for all $\chi \in \hat{G}$. 
\end{re}

The following is a result by Kronecker \cite[Section 2.3, Thm. 2]{BorevichShafarevich}.
\begin{re}[Kronecker]\label{kronecker}
	An algebraic integer all of whose conjugates have absolute values equal to $1$ is a root of unity.
\end{re}

\begin{defi}\label{defi:selfconjugate}
	Let $p$ be a prime and $v$ be a positive integer. 
	If there is an integer $j$ with $p^j \equiv -1 \bmod{v(p)}$, then $p$ is called 
	\textbf{self-conjugate} modulo $v$. A composite integer $n$ is called self-conjugate modulo $v$ if 
	every prime divisor of $n$ has this property.
\end{defi}

The reader may refer to \cite[Re. 1.3.7]{MingMingTanThesis} for the proof of the following result. 

\begin{re} \label{selfconjugate_solution2}
	Let $q$ and $p$ be distinct primes where $q$ is odd and ${\rm ord}_q(p)$ is even. 
	If $|X|^2=p^f$ for $X\in \Z[\zeta_{q^c}]$ and some positive integers $c$, $f$, 
	then $f$ is even. 
\end{re}

\begin{re}[Turyn \cite{Turyn}] \label{turyn} 
	Assume that $X \in \Z[\zeta_v]$ satisfies 
	$$X\overline{X} \equiv 0 \bmod{t^{2b}}$$
	where $b, t$ are positive integers, and $t$ is self-conjugate modulo $v$. Then 
	$$X \equiv 0 \bmod{t^b}.$$
\end{re}

\begin{re}[Ma \cite{Ma}]\label{Ma}
	Let $p$ be a prime and let $G$ be a finite abelian group with a cyclic Sylow $p$-subgroup. If $Y\in \Z[G]$ satisfies
	$\chi(Y)\equiv 0\ mod\ p^a$
	for all non-trivial characters $\chi$ of $G$,
	then there exist $X_1,X_2\in {\Z}[G]$ such that $$Y=p^aX_1+PX_2,$$
	where $P$ is the unique subgroup of order $p$ of $G$.
\end{re}

An element $X$ of $\Z[\zeta_v]$ can be expressed uniquely in terms of an integral basis of $\Q(\zeta_v)$ over $\Q$.
Note that throughout this paper, we use the integral basis $B$ as defined in the following result.
See \cite[Lem. 2.4]{Schmidt} for a proof.

\begin{re} \label{integralbasis}
	Let $u=\prod_{i=1}^{\delta(u)} q_i^{a_i}$ be the prime power decomposition of $u$. Then 
	$$B=\left\{ \prod_{i=1}^{\delta(u)} \zeta_{q_i}^{k_i}\zeta_{q_i^{a_i}}^{l_i}: 0 \leq k_i \leq q_i-2, 0 \leq l_i \leq q_i^{a_i-1}-1 \right\}$$
	is an integral basis of $\Q(\zeta_u)$ over $\Q$. 
	Suppose $X \in \Z[\zeta_u]$ has the form:
	$$X = \sum_{j=0}^{u-1} b_j \zeta_u^j,$$
	where $b_0, \ldots, b_{u-1}$ are integers with $|b_j| \leq C$ for some constant $C$. 
	Then, if we express $X$ in terms of the integral basis $B$, we will have 
	$$X = \sum_{x \in B} c_x x$$
	where the integers $c_x$ satisfy $|c_x| \leq 2^{\delta(u)}C$ for all $x \in B$.
\end{re}

The reader may refer to the proof of the following corollary in \cite[Cor. 1.3.14]{MingMingTanThesis}. 
\begin{cor}\label{cassels2}
	Let $X \in \Z[\zeta_{2^c}]$ such that $|X|^2=n$. 
	Let $B=\left\{1, \zeta_{2^c}, \ldots, \zeta_{2^c}^{2^{c-1}-1}\right\}$. 
	Write $X = \sum_{x \in B} c_x x$, where $c_x$ is an integer for each $x \in B$.
	Then, $|c_x| \leq \sqrt{n}$ for each $x \in B$. 
\end{cor}

The following notations and definitions are required in order to use the field descent method \cite{Schmidt}.
For a prime $p$ and a positive integer $t$, let $\nu_p(t)$ be defined by $p^{\nu_p(t)} || t,$ i.e., 
$p^{\nu_p(t)}$ is the highest power of $p$ dividing $t$. 
By $\mathcal{D}(t)$ we denote the set of prime divisors of $t$.

\begin{defi}\label{defi:F-value}
	Let $v, n$ be integers greater than $1$. For $q \in \mathcal{D}(n)$, let 
	\begin{equation*}
		\mu_q=\left\{
		\begin{array}{cc}
			\prod_{p \in \mathcal{D}(v) \setminus \{q\}} p & \mbox{ if $v$ is odd or $q = 2,$}\\
			4\prod_{p \in \mathcal{D}(v) \setminus \{2,q\}}p & \mbox{ otherwise. }\\
		\end{array}\right.
	\end{equation*}
	Set 
	\begin{eqnarray*}
		b(2,v,n) &=& \max_{q \in \mathcal{D}(n)\setminus\{2\}} \left\{ \nu_2(q^2-1) + \nu_2({\rm ord}_{\mu_q}(q))-1\right\} \mbox{ and}\\
		b(r,v,n) &=& \max_{q \in \mathcal{D}(n)\setminus\{r\}} \left\{ \nu_r(q^{r-1}-1) + \nu_r({\rm ord}_{\mu_q}(q))\right\}
	\end{eqnarray*}
	for primes $r > 2$ with the convention that $b(2, v, n) = 2$ if $\mathcal{D}(n)=\{2\}$ and 
	$b(r, v, n)=1$ if $\mathcal{D}(n)=\{r\}.$ 
	We define
	$$F(v,n)=\gcd \left(v, \prod_{p\in \mathcal{D}(v)} p^{b(p,v,n)}\right).$$
	Note that we define an empty product as $1$.
\end{defi}

The following result was proved in \cite[Thm. 2.2.8]{SchmidtBook}.
\begin{re}\label{Fdescent}
	Assume $X\overline{X} = n$ for $X \in \Z[\zeta_v]$ where $n$ and $v$ are positive integers. Then 
	$$X\zeta_v^j \in \Z[\zeta_{F(v,n)}]$$ for some integer $j$.
\end{re}

We use the following important consequence of \Cref{Fdescent}. 

\begin{re}[F-bound] \label{Fbound}
	Let $X \in \Z[\zeta_v]$ be of the form 
	$$X = \sum_{i=0}^{v-1}a_i\zeta_v^i$$ 
	with $|a_i| \leq C$ for some constant $C$ and assume that $n= X\overline{X}$ is an integer. 
	Then 
	$$n \leq \frac{C^2F(v,n)^2}{\varphi(F(v,n))}.$$
\end{re}

\section{Multiplier Theorems} \label{sec:Multiplier}
Multipliers were first introduced for the studies of difference sets. 
Arasu and Seberry \cite{ArasuSeberry} were the first to notice its applications to the studies of circulant weighing matrices. 
Arasu and Ma \cite{ArasuMa} extended the concept of multipliers from integer group rings to cyclotomic group rings, which helps to tackle cases where $\gcd(v,n) > 1$. However, their result only covers cyclic groups. 
We will generalize their result to abelian groups and give a simpler proof. 

\begin{defi} \label{defi:cyclotomic_multiplier}
	Let $X \in \Z[\zeta_u][G]$ where $G$ is an abelian group and 
	$u$ is a positive integer. 
	An integer $t$ with $\gcd(t, u|G|)=1$ is called a \textbf{multiplier} of 
	$X$ if $$X^{(t)}= \pm \zeta_u^e g X$$ for some $g \in G$ and $e \in \Z$.
\end{defi}

The following classical multiplier theorem is due to McFarland, see \cite[Thm. 9.1]{McFarland}.

\begin{re}[McFarland] \label{McFarland}
	Let $G$ be a finite abelian group of order $v$ and exponent $e$. 
	Let $D \in \Z[G]$ such that 
	$$DD^{(-1)} \equiv n \bmod{G}$$ where $\gcd(v,n)=1$. 
	Let $$n = p_1^{e_1} \cdots p_s^{e_s}$$
	where the $p_i$'s are distinct primes. 
	Let $t$ be an integer with $\gcd(v,t)=1$ and suppose there are integers $f_1, \cdots f_s$ such that 
	$$t \equiv p_1^{f_1} \equiv \ldots \equiv p_s^{f_s} \bmod{e}.$$
	Then $t$ is a multiplier of $D$.
	Furthermore, $D$ has a translate $D'=Dg$, $g \in G$, such that 
	$$D'^{(t)} = D'.$$
\end{re}

In the following, we present a partial generalization of Theorem 3.3 in \cite{ArasuMa}.
Note that it contains McFarland's result as a special case. 
We need the following lemma. 

\begin{lem} \label{lem: fixes1}
	Let $X \in \Z[\zeta_u][G]$ where $G$ is an abelian group and 
	$u$ is a positive integer.
	If $XX^{(-1)}=1$, then $X=\pm \zeta_u^eg$
	for some $g \in G$ and $e \in \Z$.
\end{lem}

\begin{proof}
	Write $X=\sum_{g \in G} X_gg$ with $X_g \in \Z[\zeta_u]$. 
	From $XX^{(-1)}=1$ we get $$\sum_{g \in G} |X_g|^2 =1.$$ 
	This implies
	$$\sum_{g \in G} |X_g^{\sigma}|^2 = \left(\sum_{g \in G} |X_g|^2\right)^{\sigma}=1^{\sigma}=1$$
	for all $\sigma \in {\rm Gal}(\Q(\zeta_u)/\Q)$.
	Hence $|X_g^{\sigma}|\le 1$ for all $g \in G$ and $\sigma \in {\rm Gal}(\Q(\zeta_u)/\Q)$.
	Now Result \ref{kronecker} implies that there is $g \in G$ such that $X_g=\pm \zeta_u^e$, for some integer $e$ 
	and $X_k=0$ for all $k \neq g$. 
\end{proof}

\begin{thm} [Generalized Multiplier Theorem]\label{thm:fixesn}
Let $X \in \Z[\zeta_u][G]$ where $G$ is an abelian group and 
$u$ is a positive integer.
Let $v=\lcm(u,\exp(G))$. 
Suppose that 
\begin{equation} \label{eq: fixesn} 
X X^{(-1)}=n,
\end{equation}
where $n$ is a positive integer with $\gcd(n,|G|)=1$. 
Suppose $t$ is an integer that fixes all prime ideals of $n$ in $\Z(\zeta_v)$. 
Then $t$ is a multiplier of $X$.
\end{thm}

\begin{proof}
	Write $F=X^{(t)}X^{(-1)}$. Let $\chi$ be any character
	of $G$. Then $|\chi(X)|^2 =n$ by (\ref{eq: fixesn}). Note that
	$\chi(X^{(t)}) = \chi(X)^{\sigma_t}$. 
	Since $\sigma_t$ fixes all prime ideals above
	$n$ in $\Q(\zeta_v)$,
	we conclude that $\Z[\zeta_v]\chi(X^{(t)})$ and $\Z[\zeta_v]\chi(X)$ have the
	same prime ideal factorization. 
	Since $\chi(X)\overline{\chi(X)} \equiv 0\ (\bmod\ n)$, we conclude
	$\chi(F) = \chi(X^{(t)}) \overline{\chi(X)} \equiv 0\ (\bmod\ n)$.
	Since $(n,|G|)=1$, Result \ref{inversion} implies
	$F \equiv 0\ (\bmod\ n)$, say $F=nE$ with $E\in \Z[\zeta_u][G]$. 
	Then $|\chi(E)|^2 = 1$ for all characters $\chi$ of $G$. 
	Hence $EE^{(-1)}=1$ by Result \ref{inversion} and thus $E=\pm \zeta_u^eg$
	for some integer $e$ and $g \in G$ by Lemma \ref{lem: fixes1}.
	
	\medskip
	
	We conclude $$\pm \zeta_u^egnX = EnX= FX =X^{(t)}X^{(-1)}X=X^{(t)}n.$$ 
	This implies $X^{(t)}=\pm \zeta_u^{e}gX$.
\end{proof}

\begin{cor} \label{cor:fixingmult0}
	Let $X \in \Z[\zeta_u][G]$ where $G$ is an abelian group and $u$ is a positive integer.
	Let $v=\lcm(u,\exp(G))$. 
	Suppose $t$ is a multiplier of $X$ and $X^{(t)} = \pm \zeta_u^e gX$ for some $g \in G$ and integer $e$. 
	\begin{enumerate}[{\normalfont (a)}]
	\item If $\gcd(t-1, u) = 1$, then there is a translate $Y$ of $X$ where $Y^{(t)} = \pm gY$.          
	\item If $\gcd(t-1, \exp(G))= 1$, then there is a translate $Y$ of $X$ where $Y^{(t)} = \pm \zeta_u^e Y$.
	\item If $\gcd(t-1, v)=1$, then there is a translate $Y$ of $X$ where $Y^{(t)} = \pm Y.$
	\end{enumerate}
\end{cor}
	
	\begin{proof}
	If $\gcd(t-1, u) = 1$, let $Y= \zeta_u^fX$, where $f$ is an integer such that $(1-t)f \equiv e \bmod{u}$, then we get $Y^{(t)} = \pm gY$. 
	If $\gcd(t-1, \exp(G))= 1$, let $Y= hX$ where $h \in G$ satisfies $h^{t-1} = g^{-1}$, then we get $Y^{(t)} = \pm \zeta_u^e Y$.
	If $\gcd(t-1, v)=1$, let $Y = \zeta_u^fhX$, then we get $Y^{(t)} = \pm Y.$ 
	\end{proof}

\begin{cor} \label{cor:fixingmult}
	Let $p$ be an odd prime and $u$ be a positive integer coprime to $p$. 
	Let $V$ be an abelian $p$-group of exponent $p^d$. 
	Let $X \in \Z[\zeta_u][V]$ such that $X\overline{X} = k^2$ for some positive integer $k$ coprime to $p$. 
	Suppose $t$ is a multiplier of $X$ where $t \equiv 1 \bmod{u}$. 
	Then we can replace $X$ by a translate so that $X^{(t)} = X$.
\end{cor}
	
	\begin{proof}
	Since $t - 1 \equiv 0 \bmod{u}$ and $\gcd(u,p)=1$, we deduce that $\gcd(t-1,p)=1$. It follows from \Cref{cor:fixingmult0} that we may assume 
	$X^{(t)} = z X$
	where $z=\epsilon \zeta_u^e$ for some $\epsilon \in \{ \pm 1\}$ and integer $e$.
	For a group element $g$ in $V$, let $X_g$ be the cyclotomic coefficient of $g$ in $X$. Then
	$X_{g}^{(t)} = z X_{g^t}.$ 
	Note that $t \equiv 1 \bmod{u}$ implies $X_{g}^{(t)} = X_{g}$. 
	We conclude 
	\begin{equation}\label{eq:thm158_100}
	X_{g} =z X_{g^t}.
	\end{equation}
	In particular, when $g=1$, we have $X_1=zX_1.$
	This implies either $z=1$ or $X_1=0$. 
	Let $\rho$ be the trivial character of $V$. 
	Then $\rho(X)=\sum_{g \in V} X_g.$ 
	Consider the sequence of orbits under the map $g \mapsto g^t$ on $V$, sorted in non-decreasing order of the orbit size. 
	Let $s_i$ denote the size of the $i$-th orbit. 
	Then (\ref{eq:thm158_100}) implies 
	$$\rho(X) = X_1 + \sum_{i >1} Y_i(1+z+z^2+\ldots+z^{s_i-1})$$ where $Y_i$ is the cyclotomic coefficient of an element $g_i$ in $X$, where $g_i$ belongs to the $i$-th orbit.
	In the following, we show that $z$ is a $w$-th root of unity. 
	Let $s=\ord_{p^d}(t)$ and $w=\ord_{p}(t)$.
	Using (\ref{eq:thm158_100}), we have 
	$$X_{g} = z^s X_{g^{t^s}} = z^s X_{g}$$ for all $g \in V$. 
	This implies either $z^s = 1$ or $X_{g} = 0$ for all $g \in V$. 
	The latter is impossible, as otherwise $X = 0$. 
	Hence, $z^s = 1$. 
	Since $s$ is divisible by $w$ and $\gcd(p, 2u) = 1$, we deduce $z^w =1$. 
	As a result, for any $i>1$, since $s_i$ is a multiple of $w$, we have $1+z+z^2+\ldots+z^{s_i-1}=0$ if $z$ is not equal to $1$. 
	Consequently, $\rho(X)=0$. This cannot be true because $\rho(X)\overline{\rho(X)}=k^2$. Hence, $z=1$.
	\end{proof}

\section{Results}
We divide our non-existence results on group invariant weighing matrices into three subsections based on the different approaches employed: the field descent method, the generalized multiplier theorems, and the rational idempotent/Weil number approach.

\subsection{Field Descent Method}
In the following, we apply the F-bound stated as Result \ref{Fbound} to give an upper bound on the weight of a group invariant weighing matrix. 

\begin{cor} \label{Fbound_IW}
Let $G$ be an abelian group of order $v$.
Let $w$ be the exponent of $G$ and let $h=v/w$. 
If $D$ is a $G$-invariant $IW_a(v,n)$, then $$n \leq \frac{a^2h^2F(w,n)^2}{ \varphi(F(w,n))}.$$
\end{cor}

\begin{proof}
	Let $H$ be a subgroup of $G$ of order $h$ such that $G/H$ is cyclic.
	Write $G = C_w \times H$ and let $g$ be a generator of $C_w$. 
	Write $D = \sum_{i=0}^{w-1} \sum_{h \in H}a_{i,h}g^ih$, where $|a_{i,h}| \leq a$. 
	Let $\chi$ be a character of $G$ of order $w$ which is trivial on $H$. 
	Then $\chi(D)=\sum_{i=0}^{w-1} b_{i}\zeta_u^i$ where $|b_{i}| \leq ah$. 
	By Result \ref{character_GW}, $|\chi(D)|^2=n$. 
	Now the assertion is proved using the F-bound (\Cref{Fbound}). 
\end{proof}

This result alone is enough to rule out some infinite families of group invariant weighing matrices. 
In particular, we settle the following open cases from Strassler's table and the table of group invariant weighing matrices in \cite{arasu2013group}.

\begin{thm} \label{thm_fielddescent}
$CW(v, n)$ do not exist for all parameters $(v, n)$ listed in Table \ref{tableF}. In the table, $c$, $d$, $e$ and $f$ are any non-negative integers. 
In addition, let $v$, $w$ and $n$ be as listed in Table \ref{tableF2}. Then $G$-invariant $IW(v, n)$ does not exist for any abelian group $G$ where its order is $v$ and its exponent is divisible by $w$.
Note that the last column of the Table \ref{tableF} shows the particular values of $v$ and $n$ which correspond to some previously 
open cases in Strassler's table.
The second column of Table \ref{tableF} showcases the $F$-value for each case.
\end{thm}

\begin{table}[!h]
\caption {\Cref{tableF}: $CW(v,n)$s in violation of Corollary \ref{Fbound_IW} } \label{tableF}
\centering
\begin{tabular}{|l|l|l|}
\hline
$	CW(	v	,	n	)	$ & $	F(v,n)	$ & $	CW(	v	,	n	)	$	\\	\hline
$	(	2^c 	,	7^{2e + 2}	)	$ & $	2^4	$ & $	(	128	,	49	)	$	\\	\hline
$	(	3.7^d	, 	2^{2e+6}	)	$ & $	3.7	$ & $	(	147	,	64	)	$	\\	\hline
$	(	3^c13^d 	,	3^{2f+4}	)	$ & $	3.13	$ & $	(	117	,	81	)	$	\\	\hline
$	(	2^c5^d 	,	2^{2e+2}5^{2f+2}	)	$ & $	2^2. 5	$ & $	(	160	,	100	)	$	\\	\hline
$	(	2^c11^d 	,	2^{2e+2}5^{2f+2}	)	$ & $	2^2.11	$ & $	(	176	,	100	)	$	\\	\hline
$	(	2^c3^d 	,	2^{2e+2}5^{2f+2}	)	$ & $	2^3.3	$ & $	(	192	,	100	)	$	\\	\hline
\end{tabular}
\end{table}

\begin{table}[!h]
\caption {\Cref{tableF2}: $IW(v,n)$s in violation of Corollary \ref{Fbound_IW} } \label{tableF2}
\centering
\begin{tabular}{|l|l|l|}
\hline
$	IW(	v	,	n	)	$ & $	w	$  	\\	\hline
$   (2^c, 2^{2e+6})$ & $2^{c-1}$ \\ \hline
$	(	3^c 	,	7^{2e + 2}	)	$ & $ 3^{c-1}$ 	\\	\hline
\end{tabular}
\end{table} 

Next we utilize another important concept of the field descent method: the decomposition of group ring elements into two parts, 
one part corresponding to the subfield given by the field descent and a second part corresponding to the kernel of a map from 
the integral group ring to a group ring with cyclotomic integers as coefficients. See \cite[Thm. A]{LeungSchmidt} for the details. 
The following non-existence result relies upon such a decomposition. 

\begin{thm} \label{thm105_81}
Let $a$, $b$, $e$, $w$ be positive integers where $w$ is odd
and let $p$ be an odd prime with $\gcd(p,w)=1$. 
Suppose that $\gcd(p-1,w)=1$ or that ${\rm ord}_s(p)$ is even for 
every prime divisor $s$ of $\gcd(p-1,w)$. 
If a proper $ICW_a(p^bw,p^{2e})$ exists, then $p\le 4a$ and an $ICW_{2a}(w,p^{2e})$ exists.
\end{thm}

\begin{proof} 
	Write $v=p^bw$ and let $D$ be an $ICW_a(v,p^{2e})$. 
	Let $t$ be a primitive root mod $p$ and let $\gamma$ be generator of $C_{p^b}$. 
	Since $\gcd(p^b,F(p^bw,p^{2e}))=p$, by \cite[Thm.\ 3.4]{LeungSchmidt}, there is $h\in C_v$ such that
	\begin{equation} \label{e1}
	Dh = D_w\sum_{i=1}^{p-1}\left((-1)^{\ell}g\right)^i\gamma^{p^{b-1}t^i} + YC_p
	\end{equation}
	where $D_w\in \Z[C_w]$, $\ell\in\{0,1\}$, $Y\in \Z[C_v]$, and $g$ is an
	element of $C_w$ of order dividing $p-1$. 
	
	\medskip
	
	We first show $\ell=0$. Suppose $\ell=1$. Let $\chi$ be character of
	$C_v$ with $\chi(\gamma^{p^{b-1}})=\zeta_p$ which is trivial on $C_w$. 
	Then $\chi(Dh)=\chi(D_w)\sum_{i=1}^{p-1}(-1)^i\zeta_p^{t^i}$. 
	Note that $|\sum_{i=1}^{p-1}(-1)^i\zeta_p^{t^i}|^2=p$ by
	\cite[Prop. 8.2.2, p.\ 92]{IrelandRosen}. As $|\chi(D)|^2=p^{2e}$, we conclude
	$|\chi(D_w)|^2=p^{2e-1}$, which is impossible, as $\chi(D_w)$ is an integer. 
	This proves $\ell=0$. 
	
	\medskip
	
	Next, we show $g=1$. Suppose $g\neq 1$. As the order of $g$ divides $p-1$,
	there is a prime divisor $s$ of $\gcd(p-1,w)$ which divides the order of
	$g$. Furthermore, ${\rm ord}_s(p)$ is even by assumption. 
	Let $\tau$ be a character of $C_v$ with $\tau(\gamma^{p^{b-1}})=\zeta_p$
	and $\tau(g)=\zeta_s$ such that the order of $\tau$ is $p^{b}s^m$, 
	for some positive integer $m$. Then
	$\tau(Dh)=\tau(D_w)\sum_{i=1}^{p-1}\zeta_s^i\zeta_p^{t^i}$
	by (\ref{e1}), as $\ell=0$. 
	Note that $|\sum_{i=1}^{p-1}\zeta_s^i\zeta_p^{t^i}|^2=p$ 
	by \cite[Prop. 8.2.2, p.\ 92]{IrelandRosen}.
	As $|\tau(D)|^2=p^{2e}$, we conclude $|\tau(D_w)|^2=p^{2e-1}$. 
	Note that $\tau(D_w)\in \Z[\zeta_{s^m}]$ and ${\rm ord_s(p)}$
	is even by assumption. Hence $2e-1$ is even by Result \ref{selfconjugate_solution2}, a
	contradiction. This shows $g=1$. 
	
	\medskip
	
	As $g=1$ and $\ell=0$, we get
	\begin{equation} \label{e2}
	Dh = D_w(C_p-1)+Y(C_p)=-D_w +C_p(Y + D_w)
	\end{equation}
	from (\ref{e1}). 
	Suppose that $k\in C_w$ has coefficient $c_1$ in $D_w$ with $|c_1|>2a$. 
	As $Dh$ is an $ICW_a(v,p^{2e})$, this implies that $k$ has a coefficient 
	$c_2$ in $C_p(Y + D_w)$ with $|c_2|>a$. But then the coefficient $c_3$ of $k\gamma^{p^{b-1}}$
	in $Dh$ satisfies $|c_3|>a$ by (\ref{e2}). This contradicts the 
	assumption that $D$ is an $ICW_a(v,p^{2e})$. Hence all coefficients 
	of $D_w$ are at most $2a$ in absolute value. 
	Let $\rho: \Z[C_v]\to \Z[\zeta_{p^{b}}]\Z[C_w]$ be the homomorphism
	defined by $\rho(k)=k$ for $k\in C_w$ and $\rho(\gamma)=\zeta_{p^b}$. 
	Then $\rho(Dh)=-D_w$ by (\ref{e2}). As $\rho(Dh)\rho(Dh)^{(-1)}=p^{2e}$,
	this implies that $D_w$ is an $ICW_{2a}(w,p^{2e})$. 
 
	 \medskip
	 
	 It remains to show $p\le 4a$. Suppose $p>4a$. As $D$ is proper by assumption,
	 we have $D\neq -D_w$. 
	 Write $Dh=\sum_{k\in C_v}d_kk$ and $D_w=\sum_{k\in C_v}e_kk$, where $e_k = 0$ if $k \not \in C_w$. 
	 Note that
	 \begin{equation} \label{e2a}
	\sum_{k\in C_v} d_k^2=\sum_{k\in C_v} e_k^2 = p^{2e}.
	\end{equation}
	 Suppose the coefficient of $z\in C_v$ in $C_p(Y + D_w)$ is $c$.
	This implies that the coefficient
	 of each element of $C_pz$ in $C_p(Y + D_w)$ is equal to $c$. 
	 By (\ref{e2}), we have
	 $d_k = -e_k+c$ for all $k\in C_pz$. 
	 Furthermore, as $D_w\in \Z[C_w]$,
	 there is $x\in C_pz$ such that $e_k=0$ for 
	 all $k\in C_pz$, $k\neq x$. Thus
	 $$\sum_{k\in C_pz} d_k^2 = (-e_x+c)^2 + (p-1)c^2 = pc^2 -2ce_x+e_x^2. $$
	Hence
	\begin{equation} \label{e3}
	\sum_{k\in C_pz} (d_k^2-e_k^2)=(pc^2 -2ce_x+e_x^2)-e_x^2=pc^2-2ce_x.
	\end{equation}
	We have $|e_x|\le 2a$, as $D_w$ is an $ICW_{2a}(w,p^{2e})$. 
	 As $p>4a$ by assumption, we have $|2e_x|<p$. 
	Note that the parabola $f(t)=pt^2-2te_x$ attains its minimum for $t=e_x/p$ and we have $|e_x/p|<1$.
	As $c$ is an integer, we conclude
	$$pc^2-2ce_x\ge \min\{f(0),f(-1),f(1)\}=\min\{0,p\pm 2e_x\}=0$$ with equality if and only if $c=0$. 
	Hence
	\begin{equation} \label{e4}
	\sum_{k\in C_pz} (d_k^2-e_k^2)\ge 0
	\end{equation}
	by (\ref{e3}) with equality if and only if $c=0$. 
	
	\medskip
	
	As $D$ is proper by assumption, we have $Dh\neq D_w$. Hence there is $z\in C_v$
	such that the coefficient of $z$ in $C_p(Y + D_w)$ is nonzero. 
	Hence
	$\sum_{k\in C_v} d_k^2>\sum_{k\in C_v} e_k^2$ by (\ref{e4}), which contradicts
	(\ref{e2a}). Thus $p\le 4a$.
\end{proof}

As a consequence, we have the following result. 
A weaker version of this corollary is contained in \cite[Thm. 3.1]{ArasuMa2001b}.

\begin{cor}
Let $b,e,w$ be positive integers where $w$ is odd
and let $p>3$ be a prime.
Suppose that $\gcd(p-1,w) = 1$ or $\ord_s(p)$ is even for every prime divisor $s$ of $\gcd(p-1,w)$. 
Then there is no proper $CW(p^bw,p^{2e})$. 
\end{cor}

The following corollary shows the non-existence of $CW(105,81)$, which was previously an open case in Strassler's table. 

\begin{cor}\label{cor105_81}
There is no $CW(3^a5^b7^c, 3^{2e})$ for non-negative integers $a$, $b$, $c$ and $e > 1$. 
\end{cor}

\begin{proof}
	Let $v = 5^b7^c$ and $n=3^{2e}$. 
	By Theorem \ref{thm105_81}, it is sufficient to show that $ICW_{2}(v,n)$ does not exist.	
	Suppose a proper $ICW_2(v,n)$ exists, denoted by $E$. 
	By the F-bound (Result \ref{Fbound}), we deduce that $v$ is a multiple of $35$ and $n=81$. 
	By McFarland's result, we have that $3$ is a multiplier and we may assume that 
	\begin{equation} \label{eq:35_1}
		E^{(3)} = E.
	\end{equation}
	Consider the set of orbits under the map $g \mapsto g^3$ on $C_{v}$. 
	Equation (\ref{eq:35_1}) implies that elements in the same orbit share a common coefficient in $E$. 
	There is one orbit of size $1$, and for each pair of $1 \leq i \leq b$, $1 \leq j \leq c$, there is one orbit of size $\ord_{5^i}(3)=4.5^{i-1}$, one orbit of size $\ord_{7^j}(3) = 6.7^{j-1}$ and two orbits of size $\ord_{5^i7^j}(3) = 12.5^{i-1}7^{j-1}$.
	Since the support of $E$ does not exceed $81$, elements in orbits of size greater than $81$ should have coefficients $0$ in $E$.
	Since $E$ is proper by assumption, we conclude that $b \leq 2$ and $c \leq 1$. 
	If $v=5^27$, then the support of $E$ must contain exactly one orbit of elements of order $175$. The size of this orbit is $60$.
	By considering the natural projection $\rho:C_{175}\rightarrow C_{35}$, we see that the support of $\rho(E)$ will contain $12$ elements, each with coefficient $5$.  
	Then the sum of the squares of the coefficients in $\rho(E)$ will exceed $81$. This gives a contradiction because the image $\rho(E)$ is also an integer weighing matrix with weight $81$.   
	So, $v=35$. We sort the orbits in non-decreasing order of the orbit size and let $c_i$ denote the coefficient of $g_i$ in $E$, where $g_i$ is an element in the $i$-th orbit.
	By replacing $E$ with $-E$ if necessary, we may assume that the sum of the coefficients in $E$ is $9$. 
	We derive the following equations using the relations between the weight and the coefficients of $E$:  
	\begin{equation}\label{eq:35_2}
	9 = c_1 + 4c_2 + 6c_3 + 12c_4 + 12 c_5,
	\end{equation}
	\begin{equation}\label{eq:35_3}
	81 = c_1^2 + 4c_2^2 + 6c_3^2 + 12c_4^2 + 12 c_5^2.
	\end{equation}
	It can be verified that there is no solution to (\ref{eq:35_2}) and (\ref{eq:35_3}) with $|c_i| \leq 2$ for all $i$. This gives a contradiction.
\end{proof}

\subsection{Generalized Multiplier Method}
In this section, we present a collection of non-existence results derived using the generalized multiplier theorems presented in \Cref{sec:Multiplier}. These non-existence results correspond to some previously open cases in Strassler's table and in \cite{arasu2013group}.

\begin{thm} \label{thm158_100}
Let $a$, $d$, $k$, $u$ be positive integers and let $p$ be an odd prime.
Let $G$ be an abelian group of exponent $up^d$.
If the following conditions hold, then there is no $G$-invariant $IW_a(|G|, k^2)$. 
\begin{enumerate}[{\normalfont (a)}]
\item $u$, $p$ and $k$ are pairwise coprime. 
\item $k>a$ and $k$ is self-conjugate modulo $u$.
\item There is an integer $t$ coprime to $up^d$ such that $\sigma_t$ fixes all prime ideal divisors of $k$ in $\Z[\zeta_{up^d}]$ and satisfying ${\rm ord}_p(t)>k+a$. 
\end{enumerate}
\end{thm}

\begin{proof}
	Write $G = U \times V$ where $V$ is the Sylow $p$-subgroup of $G$.
	Assume there exists a $G$-invariant $IW_a(|G|,k^2)$, denoted by $E$.
	Let $\rho$ be the natural projection of $G$ to $U$. 
	Then $\rho(E)$ is also an integer weighing matrix of weight $k^2$. 
	
	\medskip
	
	By conditions (a) and (b), we can use Result \ref{inversion} and Result \ref{turyn} to show that $\rho(E) = k$, up to equivalence. 
	By condition (c), we can use Result \ref{McFarland} to show that $t$ is a multiplier of $E$.  
	Hence, we may replace $E$ by a translate and assume that 
	\begin{equation} \label{eq:1}
	\rho(E)=k,
	\end{equation} and 
	\begin{equation}\label{eq:2}
	E^{(t)}=E.
	\end{equation}
	For each $g \in G$, let $a_g$ denote the coefficient of $g$ in $E$.
	Then (\ref{eq:1}) implies that 
	\begin{equation} \label{eq:3}
		\sum_{g \in V}a_g = k,
	\end{equation}
	and (\ref{eq:2}) implies that $$a_g=a_{g^t}$$ for all $g \in V$.   
	Consider the sequence of orbits under the map $g \mapsto g^t$ on $V$, sorted in non-decreasing order of the orbit size. 
	Let $s_i$ denote the size of the $i$-th orbit, and $c_i$ denote the coefficient of $g_i$ in $E$, where $g_i$ is an element in the $i$-th orbit. 
	Then (\ref{eq:3}) becomes $$\sum c_i s_i =k.$$
	Note that if $s_i>1$, then $s_i$ is a multiple of $w:={\rm ord}_p(t)$. 
	Hence, there is an integer $M$ such that $$c_1 + wM =k.$$ 
	Note that $|c_1| \leq a$ by definition and $k>a$ by assumption. 
	We deduce $M$ is non-zero and thus $w$ divides $k - c_1$. 
	That is $$w \leq |k - c_1| \leq k + a.$$
	This contradicts the original assumption where $w > k + a$.
	Hence, no $G$- invariant $IW_a(|G|, k^2)$ exists. 	
\end{proof}

The following result is a direct consequence of Theorem \ref{thm158_100}. We use this result to rule out some previously open cases.  

\begin{cor}\label{cor158_100}
Let $a$, $d$, $m$, $n$, $u$ be positive integers, and let $p$, $q$, $r$ be distinct primes where $p$ is odd. 
Let $$w = \gcd\left( \frac{\ord_p(q)}{\gcd(\ord_u(q),\ord_p(q))}, \frac{\ord_p(r)}{\gcd(\ord_u(r),\ord_p(r))}\right).$$
Let $H$ be an abelian group of exponent $up^d$.
Let $G$ be an abelian group such that $G / N$ is isomorphic to $H$, where $N$ is a subgroup of $G$ of order $a$.
If the following conditions hold, then there is no $G$-invariant $IW(a|H|, k^2)$, where $k = r^mq^n.$
\begin{enumerate}[{\normalfont (a)}]
\item $u$, $p$, $q$ and $r$ are pairwise coprime.
\item $k$ is self-conjugate modulo $u$.
\item $a < k < w - a$.
\end{enumerate}

\end{cor}

\begin{proof}
	Let $t$ to be an integer such that $t \equiv 1 \bmod{u}$ and $\ord_p(t) = w$. Then the assertion follows directly from Theorem \ref{thm158_100}. 
\end{proof}

\begin{cor} 
Let $k$ be a product of powers of $2$ and $5$ and $c \geq 0$.
Then there is no $CW(2^c79^d, k^2)$ for any $d \geq 1$, $k > 25$ and $2^c < k < 39-2^c$.
\end{cor}

\begin{cor} \label{cor:abelian1}
Let $d \geq 1$, $p$, $k$, $c$ and $e$ be as listed in the following.
Let $G$ be an abelian group of order $2^{c+e}p^d$ and exponent $2^cp^d$. 
Then there is no $G$-invariant $IW(|G|, k^2)$.
\begin{enumerate}
\item $p = 11$, $k = 3$, $0 \leq c \leq 2$ and $e=0$. 
\item $p = 19$, $k = 5$, $0 \leq c \leq 1$ and $e=1$. 
\item $p = 23$, $k = 3$, $0 \leq c \leq 2$ and $e = 2$, or $k=9$, $0 \leq c \leq 2$ and $e=0$, or $k = 6$, $c=0$ and $e=1$. 
\end{enumerate} 
\end{cor}

The following theorem extends \Cref{thm158_100} to the cases of $IW_a(v,n)$ where $v$ and $n$ are not coprime. 

\begin{thm} \label{thm138_36}
	Let $a$, $u$, $d$, $k$ be positive integers and $p$ be odd prime. 
	Let $G$ be an abelian group of order $up^{d'}$ and exponent $up^d$.
	If the following conditions hold, then there is no $G$-invariant $IW_a(up^{d'}, k^2)$. 
	\begin{enumerate}[{\normalfont (a)}]
		\item $\gcd(u,p)=\gcd(p,k)=1$.
		\item There exists an integer $t$ such that 
		\begin{enumerate}[{\normalfont (i)}]
			\item $\sigma_t$ fixes all prime ideal divisors of $k$ in $\Z[\zeta_{up^d}]$, 
			\item $t \equiv 1 \bmod{u}$,
			\item Let $w=\ord_p(t)$. Either $k$ is self-conjugate modulo $u$ and $2^{\delta(u)}a < k < w - 2^{\delta(u)}a$ or
			$ \frac{ua}{\sqrt{\varphi(u)}} < k < \frac{w}{2\sqrt{\varphi(u)}}.$
		\end{enumerate}
	\end{enumerate}
\end{thm}

\begin{proof}
	Write $G = U \times V$ where $V$ is the Sylow $p$-subgroup of $G$. 
	Assume there exists a $G$-invariant $IW_a(|G|,k^2)$, denoted by $E$.
	Let $\chi$ be a character of $U$ of order $u$.
	Define a homomorphism $\kappa: \Z[G] \to \Z[\zeta_{u}][V]$, such that 
	$\kappa(h)= \chi(h)$ for all $h \in U$ and $\kappa(g) = g$ for all $g \in V$.
	Let $X = \kappa(E)$.
	Then $X = \sum_{g \in V} X_g g$ is in $\Z[\zeta_u][V]$ and $X\overline{X} = k^2$. 
	By Theorem \ref{thm:fixesn} and \Cref{cor:fixingmult}, we have $X^{t} = X$ and hence $X_g = X_{g^t}.$
	Let $\rho$ be the trivial character of $V$.  Then $\rho(X)=\sum_{g \in V}X_g$ and $\rho(X)\overline{\rho(X)}=k^2$.
	Consider the sequence of orbits under the map $g \mapsto g^t$ on $V$. 
	Note that the size of a non-trivial orbit is divisible by $w$. We derive
	\begin{equation}\label{eq:thm158_100_3}
	\rho(X) = X_1 + wY
	\end{equation}
	for some $Y \in \Z[\zeta_u]$.
	
	\medskip
	
	We use the integral basis $B$ of $\Q(\zeta_u)$ over $\Q$ as defined in Result \ref{integralbasis}. 
	
	\medskip
	\noindent
	\textbf{Claim 1:}
	(a) Write $X_1 = \sum_{i=0}^{u-1} b_i \zeta_u^i$. Then $|b_i| \leq a$. \\
	(b) Write $X_1=\sum_{x \in B} c_x x$. Then $|c_{x}| \leq 2^{\delta(u)}a$ for all $x \in B$. \\
	Claim 1(a) is true because $X_1 = \chi(E \cap U)$ and each coefficient of $E$ is bounded by $a$ in absolute value. Claim 1(b) follows from Claim 1(a) and Result \ref{integralbasis}.
	
	\medskip
	\noindent
	\textbf{Claim 2:} If $k$ is self-conjugate modulo $u$ then $\rho(X) = k$.\\
	By conditions (i) and (ii), one can use Result \ref{turyn} and Result \ref{kronecker} to show that $\rho(E) = k$, up to equivalence. 
	Hence, we may replace $X$ by a translate so that $\rho(X) = k$ is satisfied as well.
	
	\medskip
	\noindent
	\textbf{Claim 3:} $Y \neq 0$.\\
	Suppose $Y = 0$. 
	Then $X_1 = \rho(X)$ from (\ref{eq:thm158_100_3}). 
	By Claim 1(a) and using the F-bound (Result \ref{Fbound}), we have $k \leq \frac{ua}{\sqrt{\varphi(u)}}$.
	Thus, by assumption, $k$ is self-conjugate modulo $u$ and $k > 2^{\delta(u)}a$. 
	Then it follows from Claim 2 that $\rho(X) = k$. Hence, $X_1 = k$.  
	Now we view $X_1$ in terms of the integral basis $B$ as shown in Claim 1(b).
	It follows from the uniqueness of a basis representation that $c_{1} = k$.
	Subsequently, by Claim 1(b), $k \leq 2^{\delta(u)}a$. This gives a contradiction and hence, $Y \neq 0$. 
	
	\medskip
	
	Up to now, we have shown that 
	$\rho(X) - X_1$ is a non-zero element in $\Z[\zeta_u]$ which is divisible by $w$. 
	By \cite[Chapter 4, Thm. 2.1]{go}, we have $|X_1| \leq k$.
	Hence $\sum (|\rho(X) - X_1|^2)^{\sigma}\leq 4k^2\varphi(u)$ where the sum runs over $\Gal(\Q(\zeta_u)/Q)$. 
	But this sum is divisible by $w^2$, so $k \geq \frac{w}{2\sqrt{\varphi(u)}}$. 
	Hence, by assumption, $k$ is self-conjugate modulo $u$ and $k < w - 2^{\delta(u)}a$.
	Recall that this implies $\rho(X) = k$, see Claim 2. 
	Expressing all terms of $\rho(X)$ in their basis representations, we deduce that 
	$k - c_{1}$ is divisible by $w$. Note that $k - c_{1} \neq 0$, for otherwise $Y=0$. 
	This implies $w \leq |k - c_{1}| \leq k + 2^{\delta(u)}a$.
	That is, $k \geq w - 2^{\delta(u)}a$. 
	This gives a contradiction.
	Hence, no $G$-invariant $IW_a(up^{d'}, k^2)$ exists. 
\end{proof}

The following result is a direct consequence of Theorem \ref{thm138_36}. It rules out some infinite families of circulant weighing matrices which include two previously open cases $CW(138, 36)$ and $CW(184, 64)$. 

\begin{cor}\label{cor138_36}
	Let $a$, $c$, $d$, $d'$, $m$, $n$ be positive integers. Let $p$, $q$ and $r$ be distinct primes where $p$ and $q$ are odd. 
	Let $$w = \gcd\left( \ord_p(q) , \frac{\ord_p(r)}{\gcd(\ord_{q^c}(r),\ord_p(r))}\right).$$
	Let $H$ be an abelian group of order $q^cp^{d'}$ and exponent $q^cp^d$. 
	Let $G$ be an abelian group and $N$ is a subgroup of $G$ of order $a$ such that $G / N$ is isomorphic to $H$.
	If $\ord_q(r)$ is even and $2a < k < w -2a$ or 
	$\frac{aq^c}{\sqrt{\varphi(q^c)}} < k < \frac{w}{2\sqrt{\varphi(q^c)}},$
	then there is no $G$-invariant $IW(aq^cp^{d'}, k^2)$ where $k = r^mq^n$.
\end{cor}

\begin{cor}
	Let $b \in \{0, 1\}$. 
	Then $CW(2^b3^c23^d, 36)$ and $CW(2^c23^d, 64)$ do not exist for any $c$, $d \geq 1$. 
\end{cor}

For $v = 2^cp^d$ and $k = 2^m3^n$, we derive the following stronger version of Theorem \ref{thm138_36}.
Specifically, we remove the self-conjugacy assumption stated in b(iii) of Theorem \ref{thm138_36}. 
Note that if $n \geq 1$, then $k$ is not self-conjugate modulo $2^c$ if $c \geq 3$. 
We need the following lemma, see \cite[Thm. 3.3.15]{MingMingTanThesis}.
\begin{lem}\label{lem:thm184_36}
	Let $m$ and $n$ be positive integers and $k = 2^m3^n$. 
	Let $Y \in \Z[\zeta_{8}]$ such that $|Y|^2 = k^2$.
	Let $B = \{1,\zeta_8, \zeta_8^2, \zeta_8^3 \}$. 
	Then there is a root of unity $\eta$ such that $Y \eta = \sum_{x \in B} d_x x$ satisfies the following: 
	\begin{enumerate}
		\item If $m > 0$, then there exists $x \in B$ such that $|d_x| > 2$.
		\item If $m = 0$ and $n \geq 2$, then either $Y\eta = \pm k$ or $|d_x| > 2$ for all $x \in B$. 
	\end{enumerate}
\end{lem}

\begin{thm} \label{thm184_36} 
	Let $m$, $n$, $c$, $d$ be positive integers, and $p > 3$ be prime.
	Define $$w = \gcd\left( \frac{\ord_p(3)}{\gcd(\ord_p(3), \ord_{2^c}(3))}, \ord_p(2)\right).$$
	Let $k = 2^m3^n$ such that $k < w - 2$ and $n \geq 2$ if $m = 0$. 
	Let $G$ be an abelian group of order $2^cp^{d'}$ and exponent $2^cp^d$.
	Then there is no $G$-invariant $IW(2^cp^{d'}, k^2)$. 
\end{thm} 

\begin{proof}
	Let $t$ be an integer such that $t \equiv 1 \bmod{2^c}$ and $\ord_p(t) = w$. 
	Then $\sigma_t$ fixes all prime ideal divisors of $k$ in $\Z[\zeta_{2^cp^d}]$.
	Following the same notations and the same arguement as in Theorem \ref{thm138_36}, we arrive at the following equation, 
	$$\rho(X) = X_1 + wY.$$
	for some $Y \in \Z[\zeta_{2^c}]$.
	Note that $\rho(X) \in \Z[\zeta_{2^c}]$ and $\rho(X) \overline{\rho(X)} = k^2$. Using Definition \ref{defi:F-value}, we see that $F(2^c, k^2)$ divides $8$ and hence by Result \ref{Fdescent}, we may replace $X$ by a translate so that $\rho(X) \in \Z[\zeta_8]$.
	Let $B$ be the integral basis of $\Q[\zeta_{2^c}]$ over $\Q$ as defined in Result \ref{integralbasis}. 
	Write $X_1 = \sum_{x \in B}c_xx$ and $M = \sum_{x \in B} M_x x$ where $c_x$ and $M_x$ are integers. 
	We have $\rho(X) =\sum_{x \in B} (c_x+wM_x)x$.
	Let $d_x = c_x + w M_x$ for each $x \in B$. 
	By \Cref{lem:thm184_36} and by replacing $X$ with a suitable translate, we deduce that there is an $x$ such that $|d_x| > 2$. 
	By definition, $|c_x| \leq 2$. 
	Hence, for this $x$, $M_x$ is non-zero and $w$ divides the non-zero element $d_x - c_{x}$.
	By \Cref{cassels2}, $|d_x| \leq k$. 
	Thus, $w \leq |d_x|+|c_{x}| \leq k+2$.
	But by assumption, $w > k + 2$.
	Hence, no $G$-invariant $IW(2^cp^{d'}, k^2)$ exists. 
\end{proof}

Theorem \ref{thm184_36} is used to rule out the following infinite families of circulant weighing matrices which include the open case $CW(184, 36)$.

\begin{cor} \label{cor184_36}
	There is no $CW(2^c23^d,36)$ for any $c \geq 1$ and $d \geq 1$.
\end{cor}

All previous theorems show that, for some parameters $v$ and $k$, if $k$ falls into a bound, then $IW(v,k)$ can not exist. The upper bound is determined by the order of a multiplier of $IW(v,k)$. 
The following deals with a special case of infinite families of circulant weighing matrices which further improve the upper bound.

\begin{thm}\label{thm190_100}
Let $a$, $u$, $d$, $k$ be positive integers and $p$ be odd prime congruent to $3$ mod $4$. 
If the following conditions hold, then there is no $ICW_a(up^d, k^2)$. 
\begin{enumerate}[{\normalfont (a)}]
\item $\gcd(u, p) = \gcd(p, k) = 1$.
\item $k > 2^{\delta(u)}a$.
\item If $d > 1$, then $p > \frac{2(k+2^{\delta(u)}a)}{p-1} + 2^{\delta(u)+1}a$, else $p > 2^{\delta(u)+1}a +1$.
\item $k$ is self-conjugate modulo $u$. 
\item There exists an integer $t$ such that 
\begin{enumerate}[{\normalfont (i)}]
\item $\sigma_t$ fixes all prime ideal divisors of $k$ in $\Z[\zeta_{up^d}]$. 
\item $t \equiv 1 \bmod{u}$.
\item ${\rm ord}_p(t) = \frac{p-1}{2}$.
\item $t^{\frac{p-1}{2}} \not \equiv 1 \bmod{p^2}$ if $d > 1$. 
\end{enumerate}
\item For each integer $M$ that satisfies $1 \leq |M| \leq 2^{\delta(u)+1}a$ and $\left | k - M\frac{p-1}{2} \right | \leq 2^{\delta(u)}a$, we have $M(4k-Mp)$ is coprime with $u$ and is square free. 
\end{enumerate}
\end{thm}

\begin{proof}
	Assume there exists an $ICW_a(up^d,k^2)$, denoted by $E$. 
	Let $\rho$ be the trivial character of $C_{p^d}$. 
	We can show that there exists $X \in \Z[\zeta_u][C_{p^d}]$ such that $X\overline{X}=k^2$, $X^{(t)}=X$ and $\rho(X)=k$. 
	Consider the sequence of orbits $D_0, D_2, \ldots, D_{2d}$ under the map $g \mapsto g^t$ on $C_{p^d}$, sorted in non-decreasing order of the orbit size. 
	Then $\rho(X)$ can be expressed as $$\rho(X)=\sum_{i=0}^{2d} Y_i|D_i|$$
	where $Y_i \in \Z[\zeta_{u}]$. 
	Since $\rho(X)=k$, by comparing the coefficients on both side of the equations in their basis representations, and using conditions (b) and (c), one can show that
	\begin{equation}\label{eq:case190_2}
		k=Y_0+w(Y_1+Y_2)
	\end{equation}
	where $Y_0$ and $Y_1+Y_2$ are integers. 
	Furthermore, $Y_1 + Y_2$ satisfies the conditions of $M$ in (f). Hence, we let $M=Y_1+Y_2$. 
	Let $\chi$ be a character of $C_{p^d}$ of order $p^d$. 
	Using the evaluation of Gauss sums (see \cite[Thm. 1, p. 75]{IrelandRosen}), 
	we can deduce that $\chi(D_1)=\frac{-1+\sqrt{-p}}{2}$, 
	$\chi(D_2) = \frac{-1-\sqrt{-p}}{2}$ and $\chi(D_i) = 0$ for all $i \geq 3$. 
	Then
	$$\chi(E) = Y_0 + Y_1\left(\frac{-1+\sqrt{-p}}{2}\right) + Y_2\left( \frac{-1-\sqrt{-p}}{2}\right).$$
	A direct computation shows that 
	$$\chi(E)\overline{\chi(E)} = \left(c_{1,0}-\frac{1}{2}M \right)^2+\frac{p}{4}(Y_1 - Y_2)^2.$$
	Substituting $\chi(E)\overline{\chi(E)}= k^2$, we deduce 
	$$(Y_1 - Y_2)^2 = \frac{4}{p}\left(k^2 - \left(c_{1,0}-\frac{1}{2}M\right)^2\right).$$
	In view of (\ref{eq:case190_2}), we obtain 
	$$(Y_1 - Y_2)^2 = M(4k - pM).$$ 
	Since $M(4k - pM)$ is coprime with $u$ and is square free by assumption, we conclude that $M(4k - pM)$ is unramified over $\Q[\zeta_u]$. 
	In other words, $M(4k - pM)$ is not a square but $(Y_1 - Y_2)^2$ is a square. 
	Hence, they cannot be equal, a contradiction. 
\end{proof}

In the following, we present some applications of Theorem \ref{thm190_100} to rule out some previously open cases. In particular, we settle the cases $CW(112, 64)$, $CW(133,100)$, $CW(184,81)$ and $CW(190, 100)$ which were previously open in Strassler's table. 

\begin{cor}
	$CW(v, k^2)$ do not exist for the following $v$ and $k$.
	\begin{enumerate}
		\item When $k = 8$, $v = 2^cp$, $c \geq 1$ and $p = 7, 19$.
		\item When $k = 9$, $v = 2^b3^cp^d$, $b \in \{0, 1\}$, $c \geq 1$, $d \geq 1$ and $p = 23, 47, 59, 71, 83$. 
		\item When $k = 10$, $v = 2^b5^cp^d$, $b \in \{0, 1\}$, $c \geq 1$, $d \geq 1$ and $p = 19, 23, 47, 59, 79, 83$.
		\item When $k = 10$, $v = 7^bp^d$, $b \in \{0, 1\}$, $d \geq 1$, and $p = 19, 23$. 
	\end{enumerate}
\end{cor}

\begin{cor} \label{cor:abelian4}
	Let $G$ be an abelian group of order $2^{c+1}11^d$ and exponent  $2^c11^d$ where $d \geq 1$ and $0 \leq c \leq 2$.
	Then there is no $G$-invariant $IW(2^{c+1}11^d , 36)$.
\end{cor}

We can further modify the conditions in Theorem \ref{thm190_100} to rule out the following infinite families of circulant weighing matrices, which includes the open case $CW(184,81)$. 

\begin{cor}\label{cor184_81}
	There is no $CW(2^c23^d, 81)$ for any $c \geq 0$ and $d \geq 1$. 
\end{cor}

\begin{proof}
	Let $t$ be an integer such that $t \equiv 3^{{\rm ord}_{2^c}(3)} \bmod {2^c23^d}$. Then all the conditions in \Cref{thm190_100} are satisfied except that $9$ is not self-conjugate modulo $2^c$.
	Let $\rho$ and $X$ be as defined in Theorem \ref{thm190_100}. 
	If we can show that $\rho(X) = 9$ up to multiplication with a root of unity, then the assertion follows from Theorem \ref{thm190_100}. 
	
	\medskip
	
	Let $B$, $c_{x}$, $d_x$ and $M_x$ be as defined in the proof of Theorem \ref{thm184_36}. 
	We see that for each $x \in B$, 
	\begin{equation} \label{eq: case184_81}
	c_{x} + 11M_x = d_x.
	\end{equation}
	Note that by Corollary \ref{cassels2}, $|d_x| \leq 9$ for all $x \in B$.
	By \Cref{lem:thm184_36}, we have either $|d_1| = 9$ and $|d_x| = 0$ for all $x \neq 1$ or $2 < |d_x| < 9$ for all $x \in B$.
	By definition, $|c_{x}| \leq 2$ for all $x \in B$. 
	In view of (\ref{eq: case184_81}), we deduce that $|d_1|=9$ and $d_x = 0$ for all $x \neq 1$. 
	That is, $\rho(X) = \pm 9$. 
\end{proof}

In the following, we use a multiplier thorem to find all solutions of $ICW_2(77,100)$, and subsequently rule out the open case $CW(154, 100)$ from Strassler's table. 
\begin{thm} \label{CW77}
	Let $E = ICW_2(77,100)$. 
	Let $\alpha$ and $\beta$ be elements in $C_{77}$ of order $7$ and $11$, respectively.
	Then 
	$$E = \pm(2D_1(-\Delta+1)+D_2(\Delta-1))^{\sigma}g $$
	where 
	$D_1 = \beta + \beta^3 + \beta^4 + \beta^5 + \beta^9$, 
	$D_2 = \beta^2 + \beta^6 + \beta^7 + \beta^8 + \beta^{10},$
	$\Delta = \alpha + \alpha^2 + \alpha^4$, and 
	$\sigma$ is an automorphism of $C_{77}$ and $g \in C_{77}$.
\end{thm}

\begin{proof}
	Define a homomorphism $\kappa: \Z[C_{77}] \to \Z[\zeta_7][C_{11}]$, such that 
	$\kappa(\alpha)= \zeta_7$ and $\kappa(\beta) = \beta$.
	Let $X = \kappa(E)$. 
	It is straightforward to show that $9$ is a multiplier and hence we can assume 
	$X^{(9)} = X.$
	\medskip 
	Consequently, we can write 
	\begin{equation} 
	X = \sum _{i \in \{0,1\} \atop j \in \{0, 1, 2\}} m_{i,j} \gamma_i D_j, \qquad -4 \leq m_{i,j} \leq 4 \mbox{ for all } i, j,
	\end{equation}
	where 
	$\gamma_0 = \zeta_7 + \zeta_7^2 + \zeta_7^4, $
	$\gamma_1 = \zeta_7^3 + \zeta_7^5 + \zeta_7^6,$ and 
	$D_0 = 1$.
	
	\medskip
	
	Let $\rho$ be the trivial character of $C_{11}$. Note that 
	$\rho(X) = \gamma_0(m_{0,0} + 5m_{0,1} + 5m_{0,2}) + \gamma_1(m_{1,0} + 5m_{1,1} + 5m_{1,2}) \in \Z[\zeta_7]$, and 
	\begin{equation}\label{eq: case77}
	\rho(X)\overline{\rho(X)}=100.
	\end{equation}
	Since $5$ is self-conjugate modulo $7$, by Result \ref{turyn}, we have $\rho(X) \equiv 0 \bmod 5$ and hence, 
	$$\gamma_0 m_{0,0} + \gamma_1 m_{1,0} \equiv 0 \bmod{5}.$$ 
	Since $m_{0,0}$ and $m_{1,0}$ are restricted to $\{0, \pm 1, \ldots, \pm 4\}$, we must have $m_{0,0} = m_{1,0} = 0$.
	Solving equation (\ref{eq: case77}), we get 
	\begin{equation}\label{eq:case77b}
	\left( m_{0,1}+m_{0,2}, m_{1,1}+m_{1,2}\right)\in \left \{\pm(1,2), \pm(2,1), \pm(2,2)\right\}.
	\end{equation}
	
	\medskip
	
	Let $\chi$ be a character of $C_{11}$ of order $11$ such that $\chi(D_1) = \frac{-1+\sqrt{-11}}{2}$ and 
	$\chi(D_2) = \frac{-1-\sqrt{-11}}{2}$.
	Then, 
	\begin{equation*}
	\chi(X) = \pm Y + \frac{\sqrt{-11}}{2}\left(\gamma_0(m_{0,1}-m_{0,2}) + \gamma_1(m_{1,1}-m_{1,2})\right) 
	\end{equation*} 
	where $Y \in \left \{\frac{\gamma_0}{2}+1, \frac{\gamma_1}{2} +1, 1 \right \}.$
	Using the solutions in (\ref{eq:case77b}) to solve $\chi(X)\overline{\chi(X)} = 100$, we have $X$ is equivalent to one of the following:  
	$$X \in \left \{ 4D_1-2D_2, 2D_1(-\gamma_0+1) + D_2(\gamma_0-1) \right \}.$$
	Using \cite[Thm. 2.2]{LamLeung}, we can deduce that 
	${\rm Ker}(\kappa) = \left\{\langle \alpha \rangle Z| Z \in \Z[C_{11}] \right \}.$
	
	\medskip
	
	Suppose $X = 4D_1-2D_2$. Then $E = 4D_1 - 2D_2 + \langle \alpha \rangle Z$ for some $Z \in \Z[C_{11}]$.
	Since the coefficients of $E$ lie between $-2$ and $2$, we can rewrite $E$ as follows:
	$$E = (4-2\langle \alpha \rangle)D_1 - 2D_2 + \langle \alpha \rangle Z',$$ 
	where $Z' \in \Z[C_{11}]$ such that ${\rm supp}(Z') \subset D_2 \cup \{1\}.$
	Let $\eta$ be the natural homomorphism from $C_{77}$ to $C_{11}$. 
	We compute $W = \eta(E) = -10D_1-2D_2+7Z'$. Since $WW^{(-1)} = 100$, 
	$$100 = \sum_{i=0}^{10} (\mbox{coeff. of } \beta^i \mbox{ in } W)^2 \geq \sum_{h \in {\rm supp}(D_1)} (\mbox{coeff. of } h \mbox{ in } W)^2 = 500.$$
	A contradiction. 
	So, $X = 2D_1(-\gamma_0+1) + D_2(\gamma_0-1)$ and hence 
	$$E = 2D_1(-\Delta+1) + D_2(\Delta-1) + \langle \alpha \rangle \sum_{i=0}^{10}b_i\beta^i $$
	where $b_i$ are integers. 
	Again, since the coefficients of $E$ lie between $-2$ and $2$, we deduce that 
	\begin{equation}\label{eq: case77c}
	b_0 \in \{0, \pm 1, \pm 2 \}, b_1 = b_3 = b_4 = b_5 = b_9 = 0, \mbox{ and } b_i \in \{0, \pm 1 \} \mbox{ for all } i \in S
	\end{equation} 
	where $S = \{2, 6, 7, 8, 10\}.$
	We compute $W = \eta(E) = -4D_1 + 2D_2 + 7\left(\sum_{i \in \{0\} \cup S} b_i \beta^i\right).$
	Again, since $WW^{(-1)}= 100$, we get 
	\begin{eqnarray*}
		100 &=& \sum_{i=0}^{10} (\mbox{coeff. of } \beta^i \mbox{ in } W)^2\\
		&=& 80 + \sum_{i \in \{0\} \cup S} (\mbox{coeff. of } \beta^i \mbox{ in } W)^2
	\end{eqnarray*}
	This gives
	\begin{eqnarray*}
		20 &=& \sum_{i \in \{0\} \cup S} (\mbox{coeff. of } \beta^i \mbox{ in } W)^2 \\
		&=& (7b_0)^2 + \sum_{i \in S} (7b_i +2)^2.
	\end{eqnarray*}
	By (\ref{eq: case77c}), we get $b_i = 0$ for all $i \in \{0\} \cup S$.
	Thus, $$E = 2D_1(-\Delta+1) + D_2(\Delta-1).$$
\end{proof}

Using Theorem \ref{CW77}, we rule out the existence of $CW(154,100)$. 
This can be proved by showing that the $ICW_2(77,100)$ can not be lifted to give $CW(154,100)$. We omit the proof here. The reader may refer to \cite[Thm.  3.3.19]{MingMingTanThesis}.

\begin{thm}\label{thm154_100}
	There is no $CW(154,100)$.
\end{thm}

\subsection{Weil numbers} \label{sub:weil}
In this subsection, we utilize rational idempotents, algebraic number theory, and computer programs such as Magma and PARI to rule out the existence of $CW(60, 36)$, $CW(120, 36)$ and $CW(155,36)$.
We exclude most of the details in the proof. 
The reader may refer to \cite[Section 3.3.3]{MingMingTanThesis} for the details.

\begin{thm} \label{thm:weil60valid}
	Let $S = CW(60,36)$ and $\chi$ be a character of $C_{60}$ of order $60$.
	Then $\chi(S)$ is equivalent to $\theta_1$ or $\theta_2$ as defined in the following:
	\begin{eqnarray*}
		\theta_1&=&-3\zeta_{60}^{15} - 2\zeta_{60}^{14} + \zeta_{60}^{11} + 2\zeta_{60}^{10} + \zeta_{60}^9 +
		4\zeta_{60}^8 + 2\zeta_{60}^6 + 5\zeta_{60}^5 + 2\zeta_{60}^4 \\ 
		&&-4\zeta_{60}^2 - \zeta_{60} - 4,\\
		\theta_2&=&-\zeta_{60}^{15} + \zeta_{60}^{14} - 3\zeta_{60}^{11} - \zeta_{60}^{10} - 3\zeta_{60}^9 - 
		2\zeta_{60}^8 - \zeta_{60}^6 + 5\zeta_{60}^5 - \zeta_{60}^4 \\
		& & + 2\zeta_{60}^2 + 3\zeta_{60}.
	\end{eqnarray*}    
\end{thm}

\begin{proof}
	Note that $\chi(S)$ satisfies $\chi(S)\overline{\chi(S)}=36$. 
	Using computer assistance, we obtain all Weil number solutions $X \in \Z[\zeta_{60}]$ to the modulus equation $X\overline{X} = 36$. We have $X$ is equivalent to one of the following: 
	$\theta_1$, $\theta_2$, $2(\zeta_{60}^{15} - 2\zeta_{60}^{11} - 2\zeta_{60}^9 + 2\zeta_{60} + 2)$, $3(\zeta_{60}^{14} +\zeta_{60}^{10} + 2\zeta_{60}^8 + \zeta_{60}^6 + \zeta_{60}^4- 2\zeta_{60}^2 - 2)$ or $6$.
	So, if $X$ is not equivalent to $\theta_1$ or $\theta_2$, then $X$ is either divisble by $2$ or is divisible by $3$.

\medskip

	Suppose $\chi(S)$ divisible by $3$.
	By Result \ref{Ma}, 
	$$S = 3X_1 + C_3X_2$$
	for some $X_1, X_2$ in $\Z[C_{60}]$. 
	But this means that the coefficients of $S$ are constant modulo $3$ on each coset of $C_3$. Since $S$ has coefficients $0, \pm 1$ only, this shows that, in fact, that the coefficients of $S$ are constant on each coset of $C_3$. Thus $S = C_3Z$ for some $Z \in \Z[C_{60}]$. But $\tau(S)=0$ for all character $\tau$ of $C_{60}$ which is trivial on $C_3$. This contradicts $S \overline{S}=36$. Hence, $\chi(S)$ is not divisible by $3$. 
	
\medskip

	Now, suppose $\chi(S)$ is divisible by $2$.
	By Result \ref{Ma}, 
	$$S = 2X_1 + C_2X_2$$
	for some $X_1, X_2$ in $\Z[C_{60}]$.
	Let $g$ be the element of order $2$ in $C_{60}$ and let $\rho : C_{60} \to C_{60}/\langle g \rangle$ be the natural epimorphism.
	By \cite[Lem. 6.2]{SchmidtSmith}, 
	$$S = (1-g)X+(1+g)Y$$ with $X, Y$ in $\Z[C_{60}]$ and $\rho(Y)$ is a $CW(30,9)$.
	But $CW(30, 9)$ does not exist, see \cite{AngArasuet}. Hence, $\chi(S)$ is not divisible by $2$. 

	\end{proof}

\begin{thm}\label{thm60_36}
	There is no $CW(60, 36)$.
\end{thm}

\begin{proof}
	Suppose $S$ is a $CW(60,36)$. 
	Let $U =\rho(S)$ where $\rho$ is the natural homomorphism from $C_{60}$ to $C_{12}$. 
	Then $S$ can be expressed as 
	\begin{equation}\label{eq:60eq}
	S = \frac{1}{5}C_5U + \frac{1}{5}(5-C_5)V,
	\end{equation}
	where 
	\begin{multline*}
	V = \frac{\alpha_5}{12}C_{12} + \frac{\alpha_{10}}{12}(2C_2-C_4)C_3+ \frac{\alpha_{15}}{12}C_4(3-C_3) + \frac{\alpha_{20}}{6}(2-C_2)(C_3) \\ + \frac{\alpha_{30}}{12}(2C_2-C_4)(3-C_3) + 
	\frac{\alpha_{60}}{6}(2-C_2)(3-C_3)
	\end{multline*} 
	with $\alpha_w \in \Z[C_w]$. 
	Let $\chi_w$ be a character of $C_{60}$ of order $w$ and let $X_w=\chi_w(\alpha_w)$. 
	We have $X_w \in \Z[\zeta_w]$ and $X_w \overline{X_w} =36$. 
	Suppose we know all such solutions $X_w$, we can determine $\alpha_w$ and hence enumerate all possibilities of $V$ in (\ref{eq:60eq}). 
	In the following, we shall characterize all solutions of $X_w$ up to equivalence. 
	From Theorem \ref{thm:weil60valid}, $X_{60}$ is equivalent to $\theta_1$ or $\theta_2$.
	With computer assistance, we deduce that 
	$X_{30}$ is equivalent to $3(1 - 2\zeta_{30}-\zeta_{30}^2+\zeta_{30}^3-2\zeta_{30}^4+\zeta_{30}^5-\zeta_{30}^7)$ or $6$, 
	$X_{20}$ is equivalent to $2(2-2\zeta_{20}^3+\zeta_{20}^5-2\zeta_{20}^7)$ or $6$, and 
	$X_{15}$ is equivalent to $3(1 - 2\zeta_{15}-\zeta_{15}^2+\zeta_{15}^3-2\zeta_{15}^4+\zeta_{15}^5-\zeta_{15}^7)$ or $6$. 
	Since $6$ is self-conjugate modulo  to both $5$ and $10$, Result \ref{turyn} 
	implies that $X_5$ and $X_{10}$ are equivalent to $6$. 
	
	\medskip
	
	In view of (\ref{eq:60eq}), we have $(5-C_5)V$ is in $\Z[C_{60}]$. However, by direct analysis of all possible combinations of $X_w$ in $V$, we find that $(5-C_5)V$ is not integral. Hence, $S$ can not be $CW(60,36)$.
	\end{proof}
	
	We remark that a different argument using intensive computer searches was presented in \cite{DjokovicKotsireas}, which yields the same conclusion.
	
	\medskip
	
	The argument in Theorem \ref{thm60_36} can be extended to get all $ICW_2(60, 36)$. 
	With computer assistance, we classified all $ICW_2(60, 36)$ up to equivalence. 
	We also found all solutions of $X\overline{X} = 36$ for $X \in \Z[\zeta_{120}]$ using computer. There are $24$ of these solutions, up to equivalence. We use these $ICW_2(60, 36)$ and Weil numbers solutions to show that no $CW(120, 36)$ exists. The reader may refer to  \cite[Thm. 3.3.26]{MingMingTanThesis} for the details.
	
\begin{thm} \label{thm120_36}
There is no $CW(120, 36)$. 
\end{thm}
	
\begin{thm} \label{thm155_36}
There is no $CW(155, 36)$.
\end{thm}
	
	\begin{proof}
	Suppose $D$ is a $CW(155, 36)$.
	Let $\chi_w$ be a character of $C_{155}$ of order $w$. 
	Since $2$ is self-conjugate modulo $5$, we have $\chi_5(D) \equiv 0 \bmod{2}$. 
	With the help of computer, we deduce that the only solution of $X \in \Z[\zeta_{31}]$ satisfying 
	$X\overline{X} =36$ is equivalent to $6$. 
	Hence, $\chi_{31}(D) \equiv 0 \bmod{2}$. 
	Furthermore, with computer assistance, we deduce that there are only two non-equivalent solutions of $X \in \Z[\zeta_{155}]$ satisfying $X\overline{X} = 36$. 
	Both of the solutions are divisible by $2$. 
	Hence, $\chi_{155}(D) \equiv 0 \bmod{2}$ as well. 
	Since $\chi_w(D) \equiv 0 \bmod{2}$ for all divisors $w$ of $155$, by \Cref{inversion}, we must have all the coefficients of $D$ divisible by $2$. 
	But this violates the definition of $D$ being $CW(155,36)$. 
	Contradiction.
	\end{proof}    

\paragraph{Acknowledgment}
I am grateful to my advisor, Bernhard Schmidt for his helpful comments and discussions. Furthermore, I would like to thank the anonymous referees for their useful suggestions. I would also like to thank Artacho et al. \cite[ Remark 4.4 and 4.5]{artacho}, for pointing out the errors in the updated Strassler's table (\Cref{table}) in an earlier version of this paper.

\bibliographystyle{abbrv}

\begin{thebibliography}{10}

\bibitem{artacho}
 Artacho, F. J. A., Campoy, R., Kotsireas, I. and Tam, M. K.
 \newblock A feasibility approach for constructing combinatorial designs of circulant type. 
 \newblock arXiv:1711.02502.
 
 \bibitem{arasudillon}
 Arasu, K. T. and  Dillon, J. F. 
 \newblock Perfect ternary arrays. 
 \newblock {\em Difference sets, sequences and their correlation properties. Springer, Dordrecht,}  1-15, 1999.

\bibitem{AngArasuet}
M.~H. Ang, K.~T. Arasu, S.~Lun~Ma, and Y.~Strassler.
\newblock Study of proper circulant weighing matrices with weight 9.
\newblock {\em Discrete Math.}, 308(13):2802--2809, 2008.

\bibitem{arasu2013group}
K.~Arasu and J.~R. Hollon.
\newblock Group developed weighing matrices.
\newblock {\em Australasian Journal of Combinatorics}, 55:205--233, 2013.

\bibitem{ArasuBayesNabavi}
K.~T. Arasu, K.~Bayes, and A.~Nabavi.
\newblock Nonexistence of two circulant weighing matrices of weight $81$.
\newblock {\em Transactions on Combinatorics}, 4(3):43--52, 2015.

\bibitem{ArasuLeunget}
K.~T. Arasu, K.~H. Leung, S.~L. Ma, A.~Nabavi, and D.~K. Ray-Chaudhuri.
\newblock Circulant weighing matrices of weight {$2\sp {2t}$}.
\newblock {\em Des. Codes Cryptogr.}, 41(1):111--123, 2006.

\bibitem{ArasuLeunget2}
K.~T. Arasu, K.~H. Leung, S.~L. Ma, A.~Nabavi, and D.~K. Ray-Chaudhuri.
\newblock Determination of all possible orders of weight 16 circulant weighing
  matrices.
\newblock {\em Finite Fields Appl.}, 12(4):498--538, 2006.

\bibitem{ArasuMa2001b}
K.~T. Arasu and S.~L. Ma.
\newblock Some new results on circulant weighing matrices.
\newblock {\em J. Algebraic Combin.}, 14(2):91--101, 2001.

\bibitem{ArasuMa}
K.~T. Arasu and S.~L. Ma.
\newblock Nonexistence of {$CW(110,100)$}.
\newblock {\em Des. Codes Cryptogr.}, 62(3):273--278, 2012.

\bibitem{ArasuAli}
K.~T. Arasu and A.~Nabavi.
\newblock Nonexistence of {${\rm CW}(154,36)$} and {${\rm CW}(170,64)$}.
\newblock {\em Discrete Math.}, 311(8-9):769--779, 2011.

\bibitem{ArasuSeberry}
K.~T. Arasu and J.~Seberry.
\newblock Circulant weighing designs.
\newblock {\em J. Combin. Des.}, 4(6):439--447, 1996.


\bibitem{DesignTheory}
T.~Beth, D.~Jungnickel, and H.~Lenz.
\newblock {\em Design theory. {V}ol. {I}}, volume~69 of {\em Encyclopedia of
  Mathematics and its Applications}.
\newblock Cambridge University Press, Cambridge, second edition, 1999.

\bibitem{BorevichShafarevich}
A.~I. Borevich and I.~R. Shafarevich.
\newblock {\em Number theory}.
\newblock Translated from the Russian by Newcomb Greenleaf. Pure and Applied
  Mathematics, Vol. 20. Academic Press, New York-London, 1966.


\bibitem{Craigen}
R.~Craigen.
\newblock The structure of weighing matrices having large weights.
\newblock {\em Des. Codes Cryptogr.}, 5(3):199--216, 1995.

\bibitem{CraigenKharaghani}
R.~Craigen and H.~Kharaghani.
\newblock Hadamard matrices from weighing matrices via signed groups.
\newblock {\em Des. Codes Cryptogr.}, 12(1):49--58, 1997.

\bibitem{Launey1984}
W.~de~Launey.
\newblock On the nonexistence of generalized weighing matrices.
\newblock {\em Ars Com.}, 21:117--132, 1984.

\bibitem{DjokovicKotsireas}
D.~Dokovic and I.~Kotsireas.
\newblock There is no circulant weighing matrix of order 60 and weight 36.
\newblock http://arxiv.org/abs/1406.1399.

\bibitem{EadesHain}
P.~Eades and R.~M. Hain.
\newblock On circulant weighing matrices.
\newblock {\em Ars Combinatoria}, 2:265--284, 1976.


\bibitem{Epstein}
L.~Epstein.
\newblock The classification of circulant weighing matrices of weight $16$ and
  odd order.
\newblock diploma thesis, Bar-Ilan University, 1998.


\bibitem{GeramitaSeberry}
A.~V. Geramita and J.~S. Wallis.
\newblock Orthogonal designs. {III}. {W}eighing matrices.
\newblock {\em Utilitas Math.}, 6:209--236, 1974.

\bibitem{go}
D.~Gorenstein.
\newblock {\em Finite groups}.
\newblock Chelsea Publishing Co., New York, second edition, 1980.

\bibitem{GysinSeberry}
M.~Gysin and J.~Seberry.
\newblock On the weighing matrices of order {$4n$} and weight {$4n-2$} and
  {$2n-1$}.
\newblock {\em Australas. J. Combin.}, 12:157--174, 1995.

\bibitem{IrelandRosen}
K.~Ireland and M.~Rosen.
\newblock {\em A classical introduction to modern number theory}, volume~84 of
  {\em Graduate Texts in Mathematics}.
\newblock Springer-Verlag, New York, second edition, 1990.


\bibitem{KoukouvinosSeberry}
C.~Koukouvinos and J.~Seberry.
\newblock Weighing matrices and their applications.
\newblock {\em J. Statist. Plann. Inference}, 62(1):91--101, 1997.


\bibitem{LamLeung}
T.~Y. Lam and K.~H. Leung.
\newblock On vanishing sums of roots of unity.
\newblock {\em J. Algebra}, 224(1):91--109, 2000.

\bibitem{LeungMa2}
K.~H. Leung and S.~L. Ma.
\newblock Proper circulant weighing matrices of weight $p^2$.
\newblock Preprint.

\bibitem{LeungMa}
K.~H. Leung and S.~L. Ma.
\newblock Proper circulant weighing matrices of weight $25$.
\newblock Preprint, 2011.

\bibitem{LeungSchmidt}
K.~H. Leung and B.~Schmidt.
\newblock The field descent method.
\newblock {\em Des. Codes Cryptogr.}, 36(2):171--188, 2005.

\bibitem{Ma}
S.~L. Ma.
\newblock {\em Polynomial addition sets}.
\newblock Ph. d thesis, Univerisity of Hong Kong, 1985.

\bibitem{McFarland}
R.~L. McFarland.
\newblock{\em On multipliers of abelian difference sets}.
\newblock ProQuest LLC, Ann Arbor, MI, 1970.
\newblock Thesis (Ph.D.)--The Ohio State University.

\bibitem{Mullin}
R.~C. Mullin.
\newblock A note on balanced weighing matrices.
\newblock In {\em Combinatorial mathematics, {III} ({P}roc. {T}hird
  {A}ustralian {C}onf., {U}niv. {Q}ueensland, {S}t. {L}ucia, 1974)}, pages
  28--41. Lecture Notes in Math., Vol. 452. Springer, Berlin, 1975.

\bibitem{Ohmori}
H.~Ohmori.
\newblock Classification of weighing matrices of order {$13$} and weight {$9$}.
\newblock {\em Discrete Math.}, 116(1-3):55--78, 1993.

\bibitem{Schmidt}
B.~Schmidt.
\newblock Cyclotomic integers and finite geometry.
\newblock {\em J. Amer. Math. Soc.}, 12(4):929--952, 1999.

\bibitem{SchmidtBook}
B.~Schmidt.
\newblock {\em Characters and cyclotomic fields in finite geometry}, volume
  1797 of {\em Lecture Notes in Mathematics}.
\newblock Springer-Verlag, Berlin, 2002.

\bibitem{SchmidtSmith}
B.~Schmidt and K.~W. Smith.
\newblock Circulant weighing matrices whose order and weight are products of
  powers of 2 and 3.
\newblock {\em J. Combin. Theory Ser. A}, 120(1):275--287, 2013.

\bibitem{SeberryYamada}
J.~Seberry and M.~Yamada.
\newblock Hadamard matrices, sequences, and block designs.
\newblock In {\em Contemporary design theory}, Wiley-Intersci. Ser. Discrete
  Math. Optim., pages 431--560. Wiley, New York, 1992.

\bibitem{StantonMullin}
R.~G. Stanton and R.~C. Mullin.
\newblock On the nonexistence of a class of circulant balanced weighting
  matrices.
\newblock {\em SIAM Journal on Applied Mathematics}, 30(1):pp. 98--102, 1976.

\bibitem{Strassler}
Y.~Strassler.
\newblock {\em The classification of circulant weighing matrices of weight
  $9$}.
\newblock diploma thesis, Bar-Ilan University, 1998.

\bibitem{MingMingTanThesis}
M.~M. Tan.
\newblock {\em Relative Difference Sets and Circulant Weighing Matrices}.
\newblock Phd thesis, Nanyang Technological University, 2014.

\bibitem{Turyn}
R.~J. Turyn.
\newblock Character sums and difference sets.
\newblock {\em Pacific J. Math.}, 15:319--346, 1965.

\bibitem{SeberryWhiteman}
J.~S. Wallis and A.~L. Whiteman.
\newblock Some results on weighing matrices.
\newblock {\em Bull. Austral. Math. Soc.}, 12(3):433--447, 1975.


\bibitem{Yorgov}
V.~Yorgov.
\newblock On the existance of certain circulant weighing matrices.
\newblock {\em J. Combin. Math. Combin. Comput.}, 86:73--85, 2013.

\end{thebibliography}

\appendix

\newpage
\section{Updated Strassler's table}
Here, we provide the most updated version of Strassler's table, see \Cref{table}.
We incorporated all latest results known so far, see \cite{ArasuMa, ArasuAli, LeungMa, LeungMa2, SchmidtSmith, Yorgov, ArasuBayesNabavi}. 
Previously open cases that we have settled in this paper are labeled with ``N''. 
Those are $CW(v,n)$ for the following pairs of $(v, n)$: 
$	(60	,	36	),	$ 
$	(120	,	36	),	$ 
$	( 138	,	36	),	$ 
$	(155	,	36	),	$ 
$	( 184	,	36	),	$ 
$	( 128	,	49	),	$ 
$	( 112	,	64	),	$ 
$	( 147	,	64	),	$ 
$	( 184	,	64	),	$ 
$	( 105	,	81	),	$ 
$	( 117	,	81	),	$ 
$	( 184	,	81	),	$ 
$	( 133	,	100	),	$ 
$	( 154	,	100	),	$ 
$	( 158	,	100	),	$ 
$	( 160	,	100	),	$ 
$	( 176	,	100	),	$ 
$	( 190	,	100	),	$ and 
$	( 192	,	100	).	$ 
In the table, we also indicate which cases of $CW(v,n)$ with $n > 25$ that can be proved to be non-existent by the results presented in this paper. 
All $CW(v, n)$s with $n \leq 25$ have been completely classified, see \cite{AngArasuet,Epstein, ArasuLeunget,ArasuLeunget2,
EadesHain,Strassler}.
We use the following labels: 
(A) for Theorem \ref{thm_fielddescent}, 
(B) for Theorem \ref{thm158_100}, 
(C) for Theorem \ref{thm138_36}, 
(D) for Theorem \ref{thm184_36}, 
(E) for Theorem \ref{thm190_100}, 
(F) for Theorem \ref{thm154_100}, 
(G) for Theorem \ref{thm105_81}
(H) for Theorem \ref{thm60_36},
(I) for Theorem \ref{thm120_36},
(J) for Theorem \ref{thm155_36}.

\begin{table}[t]																					
	\footnotesize																					
	\centering																					
	\caption {\Cref{table}: Updated Strassler's table of circulant weighing matrices} \label{table}																					

	\begin{tabular}{|p{0.5cm}||p{0.5cm}|p{0.5cm}|p{0.5cm}|p{0.5cm}|p{0.65cm}|p{1cm}|p{1.35cm}|p{1cm}|p{1cm}|p{1cm}|}																					
		\hline																					
		$s$	&	1	&	2	&	3	&	4	&	5	&	6	&	7	&	8	&	9	&	10	\\[2pt] \hline
		$v$	&		&		&		&		&		&		&		&		&		&		\\[2pt]  \hhline{|=||=|=|=|=|=|=|=|=|=|=|}
		
		1	&	Y 	&	$\cdot$ 	&	$\cdot$ 	&	$\cdot$ 	&	$\cdot$ 	&	$\cdot$ 	&	$\cdot$ 	&	$\cdot$ 	&	$\cdot$ 	&	$\cdot$ 	\\[2pt] \hline
		2	&	Y 	&	$\cdot$ 	&	$\cdot$ 	&	$\cdot$ 	&	$\cdot$ 	&	$\cdot$ 	&	$\cdot$ 	&	$\cdot$ 	&	$\cdot$ 	&	$\cdot$ 	\\[2pt] \hline
		3	&	Y 	&	$\cdot$ 	&	$\cdot$ 	&	$\cdot$ 	&	$\cdot$ 	&	$\cdot$ 	&	$\cdot$ 	&	$\cdot$ 	&	$\cdot$ 	&	$\cdot$ 	\\[2pt] \hline
		4	&	Y 	&	Y 	&	$\cdot$ 	&	$\cdot$ 	&	$\cdot$ 	&	$\cdot$ 	&	$\cdot$ 	&	$\cdot$ 	&	$\cdot$ 	&	$\cdot$ 	\\[2pt] \hline
		5	&	Y 	&	$\cdot$ 	&	$\cdot$ 	&	$\cdot$ 	&	$\cdot$ 	&	$\cdot$ 	&	$\cdot$ 	&	$\cdot$ 	&	$\cdot$ 	&	$\cdot$ 	\\[2pt] \hline
		6	&	Y 	&	Y 	&	$\cdot$ 	&	$\cdot$ 	&	$\cdot$ 	&	$\cdot$ 	&	$\cdot$ 	&	$\cdot$ 	&	$\cdot$ 	&	$\cdot$ 	\\[2pt] \hline
		7	&	Y 	&	Y 	&	$\cdot$ 	&	$\cdot$ 	&	$\cdot$ 	&	$\cdot$ 	&	$\cdot$ 	&	$\cdot$ 	&	$\cdot$ 	&	$\cdot$ 	\\[2pt] \hline
		8	&	Y 	&	Y 	&	$\cdot$ 	&	$\cdot$ 	&	$\cdot$ 	&	$\cdot$ 	&	$\cdot$ 	&	$\cdot$ 	&	$\cdot$ 	&	$\cdot$ 	\\[2pt] \hline
		9	&	Y 	&	$\cdot$ 	&	$\cdot$ 	&	$\cdot$ 	&	$\cdot$ 	&	$\cdot$ 	&	$\cdot$ 	&	$\cdot$ 	&	$\cdot$ 	&	$\cdot$ 	\\[2pt] \hline
		10	&	Y 	&	Y 	&	$\cdot$ 	&	$\cdot$ 	&	$\cdot$ 	&	$\cdot$ 	&	$\cdot$ 	&	$\cdot$ 	&	$\cdot$ 	&	$\cdot$ 	\\[2pt] \hline
		11	&	Y 	&	$\cdot$ 	&	$\cdot$ 	&	$\cdot$ 	&	$\cdot$ 	&	$\cdot$ 	&	$\cdot$ 	&	$\cdot$ 	&	$\cdot$ 	&	$\cdot$ 	\\[2pt] \hline
		12	&	Y 	&	Y 	&	$\cdot$ 	&	$\cdot$ 	&	$\cdot$ 	&	$\cdot$ 	&	$\cdot$ 	&	$\cdot$ 	&	$\cdot$ 	&	$\cdot$ 	\\[2pt] \hline
		13	&	Y 	&	$\cdot$ 	&	Y 	&	$\cdot$ 	&	$\cdot$ 	&	$\cdot$ 	&	$\cdot$ 	&	$\cdot$ 	&	$\cdot$ 	&	$\cdot$ 	\\[2pt] \hline
		14	&	Y 	&	Y 	&	$\cdot$ 	&	$\cdot$ 	&	$\cdot$ 	&	$\cdot$ 	&	$\cdot$ 	&	$\cdot$ 	&	$\cdot$ 	&	$\cdot$ 	\\[2pt] \hline
		15	&	Y 	&	$\cdot$ 	&	$\cdot$ 	&	$\cdot$ 	&	$\cdot$ 	&	$\cdot$ 	&	$\cdot$ 	&	$\cdot$ 	&	$\cdot$ 	&	$\cdot$ 	\\[2pt] \hline
		16	&	Y 	&	Y 	&	$\cdot$ 	&	$\cdot$ 	&	$\cdot$ 	&	$\cdot$ 	&	$\cdot$ 	&	$\cdot$ 	&	$\cdot$ 	&	$\cdot$ 	\\[2pt] \hline
		17	&	Y 	&	$\cdot$ 	&	$\cdot$ 	&	$\cdot$ 	&	$\cdot$ 	&	$\cdot$ 	&	$\cdot$ 	&	$\cdot$ 	&	$\cdot$ 	&	$\cdot$ 	\\[2pt] \hline
		18	&	Y 	&	Y 	&	$\cdot$ 	&	$\cdot$ 	&	$\cdot$ 	&	$\cdot$ 	&	$\cdot$ 	&	$\cdot$ 	&	$\cdot$ 	&	$\cdot$ 	\\[2pt] \hline
		19	&	Y 	&	$\cdot$ 	&	$\cdot$ 	&	$\cdot$ 	&	$\cdot$ 	&	$\cdot$ 	&	$\cdot$ 	&	$\cdot$ 	&	$\cdot$ 	&	$\cdot$ 	\\[2pt] \hline
		20	&	Y 	&	Y 	&	$\cdot$ 	&	$\cdot$ 	&	$\cdot$ 	&	$\cdot$ 	&	$\cdot$ 	&	$\cdot$ 	&	$\cdot$ 	&	$\cdot$ 	\\[2pt] \hline
		21	&	Y 	&	Y 	&	$\cdot$ 	&	Y 	&	$\cdot$ 	&	$\cdot$ 	&	$\cdot$ 	&	$\cdot$ 	&	$\cdot$ 	&	$\cdot$ 	\\[2pt] \hline
		22	&	Y 	&	Y 	&	$\cdot$ 	&	$\cdot$ 	&	$\cdot$ 	&	$\cdot$ 	&	$\cdot$ 	&	$\cdot$ 	&	$\cdot$ 	&	$\cdot$ 	\\[2pt] \hline
		23	&	Y 	&	$\cdot$ 	&	$\cdot$ 	&	$\cdot$ 	&	$\cdot$ 	&	$\cdot$ 	&	$\cdot$ 	&	$\cdot$ 	&	$\cdot$ 	&	$\cdot$ 	\\[2pt] \hline
		24	&	Y 	&	Y 	&	Y 	&	$\cdot$ 	&	$\cdot$ 	&	$\cdot$ 	&	$\cdot$ 	&	$\cdot$ 	&	$\cdot$ 	&	$\cdot$ 	\\[2pt] \hline
		25	&	Y 	&	$\cdot$ 	&	$\cdot$ 	&	$\cdot$ 	&	$\cdot$ 	&	$\cdot$ 	&	$\cdot$ 	&	$\cdot$ 	&	$\cdot$ 	&	$\cdot$ 	\\[2pt] \hline
		26	&	Y 	&	Y 	&	Y 	&	$\cdot$ 	&	$\cdot$ 	&	$\cdot$ 	&	$\cdot$ 	&	$\cdot$ 	&	$\cdot$ 	&	$\cdot$ 	\\[2pt] \hline
		27	&	Y 	&	$\cdot$ 	&	$\cdot$ 	&	$\cdot$ 	&	$\cdot$ 	&	$\cdot$ 	&	$\cdot$ 	&	$\cdot$ 	&	$\cdot$ 	&	$\cdot$ 	\\[2pt] \hline
		28	&	Y 	&	Y 	&	$\cdot$ 	&	Y 	&	$\cdot$ 	&	$\cdot$ 	&	$\cdot$ 	&	$\cdot$ 	&	$\cdot$ 	&	$\cdot$ 	\\[2pt] \hline
		29	&	Y 	&	$\cdot$ 	&	$\cdot$ 	&	$\cdot$ 	&	$\cdot$ 	&	$\cdot$ 	&	$\cdot$ 	&	$\cdot$ 	&	$\cdot$ 	&	$\cdot$ 	\\[2pt] \hline
		30	&	Y 	&	Y 	&	$\cdot$ 	&	$\cdot$ 	&	$\cdot$ 	&	$\cdot$ 	&	$\cdot$ 	&	$\cdot$ 	&	$\cdot$ 	&	$\cdot$ 	\\[2pt] \hline

	\end{tabular}																					
	\caption*{\\Existence (Y), nonexistence ($\cdot$), nonexistence found by computer (*), nonexistence shown in this																					
		paper (N), and open cases (?) for $CW(v, s^2)$}																					
\end{table}

\begin{table}[t]																					
	\footnotesize																					
	\centering																					
	\caption*{\Cref{table} (continued)}																					

	\begin{tabular}{|p{0.5cm}||p{0.5cm}|p{0.5cm}|p{0.5cm}|p{0.5cm}|p{0.65cm}|p{1cm}|p{1.35cm}|p{1cm}|p{1cm}|p{1cm}|}																					
		\hline																					
		$s$	&	1	&	2	&	3	&	4	&	5	&	6	&	7	&	8	&	9	&	10	\\[2pt] \hline
		$v$	&		&		&		&		&		&		&		&		&		&		\\[2pt]  \hhline{|=||=|=|=|=|=|=|=|=|=|=|}
		
		31	&	Y 	&	$\cdot$ 	&	$\cdot$ 	&	Y 	&	Y 	&	$\cdot$ 	&	$\cdot$ 	&	$\cdot$ 	&	$\cdot$ 	&	$\cdot$ 	\\[2pt] \hline
		32	&	Y 	&	Y 	&	$\cdot$ 	&	$\cdot$ 	&	*	&	$\cdot$ 	&	$\cdot$ 	&	$\cdot$ 	&	$\cdot$ 	&	$\cdot$ 	\\[2pt] \hline
		33	&	Y 	&	$\cdot$ 	&	$\cdot$ 	&	$\cdot$ 	&	Y 	&	$\cdot$ 	&	$\cdot$ 	&	$\cdot$ 	&	$\cdot$ 	&	$\cdot$ 	\\[2pt] \hline
		34	&	Y 	&	Y 	&	$\cdot$ 	&	$\cdot$ 	&	$\cdot$ 	&	$\cdot$ 	&	$\cdot$ 	&	$\cdot$ 	&	$\cdot$ 	&	$\cdot$ 	\\[2pt] \hline
		35	&	Y 	&	Y 	&	$\cdot$ 	&	$\cdot$ 	&	$\cdot$ 	&	$\cdot$ 	&	$\cdot$ 	&	$\cdot$ 	&	$\cdot$ 	&	$\cdot$ 	\\[2pt] \hline
		36	&	Y 	&	Y 	&	$\cdot$ 	&	$\cdot$ 	&	$\cdot$ 	&	$\cdot$ 	&	$\cdot$ 	&	$\cdot$ 	&	$\cdot$ 	&	$\cdot$ 	\\[2pt] \hline
		37	&	Y 	&	$\cdot$ 	&	$\cdot$ 	&	$\cdot$ 	&	$\cdot$ 	&	$\cdot$  \hyperref[thm184_36]{(F)}	&	$\cdot$ 	&	$\cdot$ 	&	$\cdot$ 	&	$\cdot$ 	\\[2pt] \hline
		38	&	Y 	&	Y 	&	$\cdot$ 	&	$\cdot$ 	&	$\cdot$ 	&	$\cdot$  \hyperref[thm184_36]{(F)}	&	$\cdot$ 	&	$\cdot$ 	&	$\cdot$ 	&	$\cdot$ 	\\[2pt] \hline
		39	&	Y 	&	$\cdot$ 	&	Y 	&	$\cdot$ 	&	$\cdot$ 	&	$\cdot$ 	&	$\cdot$ 	&	$\cdot$ 	&	$\cdot$ 	&	$\cdot$ 	\\[2pt] \hline
		40	&	Y 	&	Y 	&	$\cdot$ 	&	$\cdot$ 	&	*	&	$\cdot$ 	&	$\cdot$ 	&	$\cdot$ 	&	$\cdot$ 	&	$\cdot$ 	\\[2pt] \hline
		41	&	Y 	&	$\cdot$ 	&	$\cdot$ 	&	$\cdot$ 	&	$\cdot$ 	&	$\cdot$ 	&	$\cdot$ 	&	$\cdot$ 	&	$\cdot$ 	&	$\cdot$ 	\\[2pt] \hline
		42	&	Y 	&	Y 	&	$\cdot$ 	&	Y 	&	$\cdot$ 	&	$\cdot$ 	&	$\cdot$ 	&	$\cdot$ 	&	$\cdot$ 	&	$\cdot$ 	\\[2pt] \hline
		43	&	Y 	&	$\cdot$ 	&	$\cdot$ 	&	$\cdot$ 	&	$\cdot$ 	&	$\cdot$  \hyperref[thm184_36]{(F)}	&	$\cdot$ 	&	$\cdot$ 	&	$\cdot$ 	&	$\cdot$ 	\\[2pt] \hline
		44	&	Y 	&	Y 	&	$\cdot$ 	&	$\cdot$ 	&	$\cdot$ 	&	$\cdot$  \hyperref[thm190_100]{(E)}	&	$\cdot$ 	&	$\cdot$ 	&	$\cdot$ 	&	$\cdot$ 	\\[2pt] \hline
		45	&	Y 	&	$\cdot$ 	&	$\cdot$ 	&	$\cdot$ 	&	$\cdot$ 	&	$\cdot$ 	&	$\cdot$ 	&	$\cdot$ 	&	$\cdot$ 	&	$\cdot$ 	\\[2pt] \hline
		46	&	Y 	&	Y 	&	$\cdot$ 	&	$\cdot$ 	&	$\cdot$ 	&	$\cdot$  \hyperref[thm184_36]{(F)}	&	$\cdot$ 	&	$\cdot$ 	&	$\cdot$ 	&	$\cdot$ 	\\[2pt] \hline
		47	&	Y 	&	$\cdot$ 	&	$\cdot$ 	&	$\cdot$ 	&	$\cdot$ 	&	$\cdot$  \hyperref[thm184_36]{(F)}	&	$\cdot$ 	&	$\cdot$ 	&	$\cdot$ 	&	$\cdot$ 	\\[2pt] \hline
		48	&	Y 	&	Y 	&	Y 	&	$\cdot$ 	&	*	&	Y \cite{SchmidtSmith}	&	$\cdot$ 	&	$\cdot$ 	&	$\cdot$ 	&	$\cdot$ 	\\[2pt] \hline
		49	&	Y 	&	Y 	&	$\cdot$ 	&	$\cdot$ 	&	$\cdot$ 	&	$\cdot$ 	&	$\cdot$  \cite{LeungMa2}	&	$\cdot$ 	&	$\cdot$ 	&	$\cdot$ 	\\[2pt] \hline
		50	&	Y 	&	Y 	&	$\cdot$ 	&	$\cdot$ 	&	$\cdot$ 	&	$\cdot$ 	&	$\cdot$ 	&	$\cdot$ 	&	$\cdot$ 	&	$\cdot$ 	\\[2pt] \hline
		51	&	Y 	&	$\cdot$ 	&	$\cdot$ 	&	$\cdot$ 	&	$\cdot$ 	&	$\cdot$ 	&	$\cdot$  \hyperref[thm158_100]{(C)}	&	$\cdot$ 	&	$\cdot$ 	&	$\cdot$ 	\\[2pt] \hline
		52	&	Y 	&	Y 	&	Y 	&	$\cdot$ 	&	$\cdot$ 	&	Y 	&	$\cdot$  \hyperref[thm158_100]{(C)}	&	$\cdot$ 	&	$\cdot$ 	&	$\cdot$ 	\\[2pt] \hline
		53	&	Y 	&	$\cdot$ 	&	$\cdot$ 	&	$\cdot$ 	&	$\cdot$ 	&	$\cdot$  \hyperref[thm184_36]{(F)}	&	$\cdot$  \hyperref[thm158_100]{(C)}	&	$\cdot$ 	&	$\cdot$ 	&	$\cdot$ 	\\[2pt] \hline
		54	&	Y 	&	Y 	&	$\cdot$ 	&	$\cdot$ 	&	$\cdot$ 	&	$\cdot$ 	&	* \hyperref[thm_fielddescent]{(A)}	&	$\cdot$ 	&	$\cdot$ 	&	$\cdot$ 	\\[2pt] \hline
		55	&	Y 	&	$\cdot$ 	&	$\cdot$ 	&	$\cdot$ 	&	$\cdot$ 	&	$\cdot$  \hyperref[thm190_100]{(E)}	&	$\cdot$  \hyperref[thm158_100]{(C)}	&	$\cdot$ 	&	$\cdot$ 	&	$\cdot$ 	\\[2pt] \hline
		56	&	Y 	&	Y 	&	$\cdot$ 	&	Y 	&	$\cdot$ 	&	$\cdot$ 	&	$\cdot$  \cite{LeungMa2}	&	$\cdot$ 	&	$\cdot$ 	&	$\cdot$ 	\\[2pt] \hline
		57	&	Y 	&	$\cdot$ 	&	$\cdot$ 	&	$\cdot$ 	&	$\cdot$ 	&	$\cdot$  \hyperref[thm158_100]{(C)}	&	Y 	&	$\cdot$ 	&	$\cdot$ 	&	$\cdot$ 	\\[2pt] \hline
		58	&	Y 	&	Y 	&	$\cdot$ 	&	$\cdot$ 	&	$\cdot$ 	&	$\cdot$  \hyperref[thm184_36]{(F)}	&	$\cdot$ 	&	$\cdot$ 	&	$\cdot$ 	&	$\cdot$ 	\\[2pt] \hline
		59	&	Y 	&	$\cdot$ 	&	$\cdot$ 	&	$\cdot$ 	&	$\cdot$ 	&	$\cdot$  \hyperref[thm184_36]{(F)}	&	$\cdot$  \hyperref[thm158_100]{(C)}	&	$\cdot$ 	&	$\cdot$ 	&	$\cdot$ 	\\[2pt] \hline
		60	&	Y 	&	Y 	&	$\cdot$ 	&	$\cdot$ 	&	$\cdot$ 	&	N \hyperref[thm60_36]{(H)}	&	$\cdot$ 	&	$\cdot$ 	&	$\cdot$ 	&	$\cdot$ 	\\[2pt] \hline

	\end{tabular}																					
	\caption*{\\Existence (Y), nonexistence ($\cdot$), nonexistence found by computer (*), nonexistence shown in this																					
		paper (N), and open cases (?) for $CW(v, s^2)$}																					
\end{table}

\begin{table}[t]																					
	\footnotesize																					
	\centering																					
	\caption*{\Cref{table} (continued)}																					

	\begin{tabular}{|p{0.5cm}||p{0.5cm}|p{0.5cm}|p{0.5cm}|p{0.5cm}|p{0.65cm}|p{1cm}|p{1.35cm}|p{1cm}|p{1cm}|p{1cm}|}																					
		
		\hline																					
		$s$	&	1	&	2	&	3	&	4	&	5	&	6	&	7	&	8	&	9	&	10	\\[2pt] \hline
		$v$	&		&		&		&		&		&		&		&		&		&		\\[2pt]  \hhline{|=||=|=|=|=|=|=|=|=|=|=|}
		
		61	&	Y 	&	$\cdot$ 	&	$\cdot$ 	&	$\cdot$ 	&	$\cdot$ 	&	$\cdot$  \hyperref[thm184_36]{(F)}	&	$\cdot$  \hyperref[thm158_100]{(C)}	&	$\cdot$ 	&	$\cdot$ 	&	$\cdot$ 	\\[2pt] \hline
		62	&	Y 	&	Y 	&	$\cdot$ 	&	Y 	&	Y 	&	$\cdot$ 	&	$\cdot$  \hyperref[thm158_100]{(C)}	&	$\cdot$ 	&	$\cdot$ 	&	$\cdot$ 	\\[2pt] \hline
		63	&	Y 	&	Y 	&	$\cdot$ 	&	Y 	&	$\cdot$ 	&	$\cdot$ 	&	N \cite{LeungMa2}  \hyperref[thm_fielddescent]{(A)}	&	$\cdot$ 	&	$\cdot$ 	&	$\cdot$ 	\\[2pt] \hline
		64	&	Y 	&	Y 	&	$\cdot$ 	&	$\cdot$ 	&	*	&	$\cdot$ 	&	* \hyperref[thm_fielddescent]{(A)}	&	$\cdot$ 	&	$\cdot$ 	&	$\cdot$ 	\\[2pt] \hline
		65	&	Y 	&	$\cdot$ 	&	Y 	&	$\cdot$ 	&	$\cdot$ 	&	$\cdot$ 	&	$\cdot$  \hyperref[thm158_100]{(C)}	&	$\cdot$  \hyperref[thm158_100]{(C)}	&	$\cdot$ 	&	$\cdot$ 	\\[2pt] \hline
		66	&	Y 	&	Y 	&	$\cdot$ 	&	$\cdot$ 	&	Y 	&	$\cdot$  \hyperref[thm190_100]{(E)}	&	$\cdot$ 	&	$\cdot$ 	&	$\cdot$ 	&	$\cdot$ 	\\[2pt] \hline
		67	&	Y 	&	$\cdot$ 	&	$\cdot$ 	&	$\cdot$ 	&	$\cdot$ 	&	$\cdot$  \hyperref[thm184_36]{(F)}	&	$\cdot$  \hyperref[thm158_100]{(C)}	&	$\cdot$  \hyperref[thm184_36]{(F)}	&	$\cdot$ 	&	$\cdot$ 	\\[2pt] \hline
		68	&	Y 	&	Y 	&	$\cdot$ 	&	$\cdot$ 	&	$\cdot$ 	&	$\cdot$ 	&	$\cdot$  \hyperref[thm158_100]{(C)}	&	$\cdot$ 	&	$\cdot$ 	&	$\cdot$ 	\\[2pt] \hline
		69	&	Y 	&	$\cdot$ 	&	$\cdot$ 	&	$\cdot$ 	&	$\cdot$ 	&	$\cdot$  \hyperref[thm158_100]{(C)}	&	$\cdot$  \hyperref[thm158_100]{(C)}	&	$\cdot$  \hyperref[thm158_100]{(C)}	&	$\cdot$ 	&	$\cdot$ 	\\[2pt] \hline
		70	&	Y 	&	Y 	&	$\cdot$ 	&	Y 	&	$\cdot$ 	&	*	&	$\cdot$  \cite{LeungMa2}	&	$\cdot$ 	&	$\cdot$ 	&	$\cdot$ 	\\[2pt] \hline
		71	&	Y 	&	$\cdot$ 	&	$\cdot$ 	&	$\cdot$ 	&	Y 	&	$\cdot$  \hyperref[thm184_36]{(F)}	&	$\cdot$  \hyperref[thm158_100]{(C)}	&	$\cdot$  \hyperref[thm184_36]{(F)}	&	$\cdot$ 	&	$\cdot$ 	\\[2pt] \hline
		72	&	Y 	&	Y 	&	Y 	&	$\cdot$ 	&	$\cdot$ 	&	$\cdot$ 	&	*	&	$\cdot$ 	&	$\cdot$ 	&	$\cdot$ 	\\[2pt] \hline
		73	&	Y 	&	$\cdot$ 	&	$\cdot$ 	&	$\cdot$ 	&	$\cdot$ 	&	$\cdot$ 	&	$\cdot$  \hyperref[thm158_100]{(C)}	&	Y 	&	$\cdot$ 	&	$\cdot$ 	\\[2pt] \hline
		74	&	Y 	&	Y 	&	$\cdot$ 	&	$\cdot$ 	&	$\cdot$ 	&	$\cdot$  \hyperref[thm184_36]{(F)}	&	$\cdot$  \hyperref[thm158_100]{(C)}	&	$\cdot$  \hyperref[thm184_36]{(F)}	&	$\cdot$ 	&	$\cdot$ 	\\[2pt] \hline
		75	&	Y 	&	$\cdot$ 	&	$\cdot$ 	&	$\cdot$ 	&	$\cdot$ 	&	$\cdot$ 	&	$\cdot$ 	&	$\cdot$ 	&	$\cdot$ 	&	$\cdot$ 	\\[2pt] \hline
		76	&	Y 	&	Y 	&	$\cdot$ 	&	$\cdot$ 	&	$\cdot$ 	&	$\cdot$  \hyperref[thm184_36]{(F)}	&	*	&	$\cdot$  \hyperref[thm184_36]{(F)}	&	$\cdot$ 	&	$\cdot$ 	\\[2pt] \hline
		77	&	Y 	&	Y 	&	$\cdot$ 	&	$\cdot$ 	&	$\cdot$ 	&	$\cdot$ 	&	$\cdot$  \cite{LeungMa2}	&	$\cdot$ 	&	$\cdot$ 	&	$\cdot$ 	\\[2pt] \hline
		78	&	Y 	&	Y 	&	Y 	&	$\cdot$ 	&	$\cdot$ 	&	Y 	&	$\cdot$  \hyperref[thm158_100]{(C)}	&	$\cdot$  \hyperref[thm158_100]{(C)}	&	$\cdot$ 	&	$\cdot$ 	\\[2pt] \hline
		79	&	Y 	&	$\cdot$ 	&	$\cdot$ 	&	$\cdot$ 	&	$\cdot$ 	&	$\cdot$  \hyperref[thm184_36]{(F)}	&	$\cdot$  \hyperref[thm158_100]{(C)}	&	$\cdot$  \hyperref[thm184_36]{(F)}	&	$\cdot$ 	&	$\cdot$ 	\\[2pt] \hline
		80	&	Y 	&	Y 	&	$\cdot$ 	&	$\cdot$ 	&	*	&	$\cdot$ 	&	*	&	$\cdot$ 	&	$\cdot$ 	&	$\cdot$ 	\\[2pt] \hline
		81	&	Y 	&	$\cdot$ 	&	$\cdot$ 	&	$\cdot$ 	&	$\cdot$ 	&	$\cdot$ 	&	$\cdot$ 	&	$\cdot$ 	&	$\cdot$  \hyperref[cor105_81]{(B)}	&	$\cdot$ 	\\[2pt] \hline
		82	&	Y 	&	Y 	&	$\cdot$ 	&	$\cdot$ 	&	$\cdot$ 	&	$\cdot$ 	&	$\cdot$  \hyperref[thm158_100]{(C)}	&	$\cdot$  \hyperref[thm184_36]{(F)}	&	$\cdot$ 	&	$\cdot$ 	\\[2pt] \hline
		83	&	Y 	&	$\cdot$ 	&	$\cdot$ 	&	$\cdot$ 	&	$\cdot$ 	&	$\cdot$  \hyperref[thm184_36]{(F)}	&	$\cdot$  \hyperref[thm158_100]{(C)}	&	$\cdot$  \hyperref[thm184_36]{(F)}	&	$\cdot$  \hyperref[thm184_36]{(F)}	&	$\cdot$ 	\\[2pt] \hline
		84	&	Y 	&	Y 	&	$\cdot$ 	&	Y 	&	$\cdot$ 	&	$\cdot$ 	&	$\cdot$  \cite{LeungMa2}	&	Y 	&	$\cdot$ 	&	$\cdot$ 	\\[2pt] \hline
		85	&	Y 	&	$\cdot$ 	&	$\cdot$ 	&	$\cdot$ 	&	$\cdot$ 	&	$\cdot$ 	&	$\cdot$  \hyperref[thm158_100]{(C)}	&	$\cdot$ 	&	$\cdot$  \hyperref[thm158_100]{(C)}	&	$\cdot$ 	\\[2pt] \hline
		86	&	Y 	&	Y 	&	$\cdot$ 	&	$\cdot$ 	&	$\cdot$ 	&	$\cdot$  \hyperref[thm184_36]{(F)}	&	$\cdot$ 	&	$\cdot$  \hyperref[thm184_36]{(F)}	&	$\cdot$  \hyperref[thm184_36]{(F)}	&	$\cdot$ 	\\[2pt] \hline
		87	&	Y 	&	$\cdot$ 	&	$\cdot$ 	&	$\cdot$ 	&	$\cdot$ 	&	$\cdot$  \hyperref[thm158_100]{(C)}	&	Y 	&	$\cdot$  \hyperref[thm158_100]{(C)}	&	$\cdot$  \hyperref[thm158_100]{(C)}	&	$\cdot$ 	\\[2pt] \hline
		88	&	Y 	&	Y 	&	$\cdot$ 	&	$\cdot$ 	&	*	&	$\cdot$  \hyperref[thm190_100]{(E)}	&	$\cdot$  \hyperref[thm158_100]{(C)}	&	$\cdot$ \cite{ArasuBayesNabavi} \hyperref[thm190_100]{(E)}	&	$\cdot$ 	&	$\cdot$ 	\\[2pt] \hline
		89	&	Y 	&	$\cdot$ 	&	$\cdot$ 	&	$\cdot$ 	&	$\cdot$ 	&	$\cdot$  \hyperref[thm184_36]{(F)}	&	$\cdot$  \hyperref[thm158_100]{(C)}	&	$\cdot$  \hyperref[thm184_36]{(F)}	&	$\cdot$  \hyperref[thm184_36]{(F)}	&	$\cdot$ 	\\[2pt] \hline
		90	&	Y 	&	Y 	&	$\cdot$ 	&	$\cdot$ 	&	$\cdot$ 	&	$\cdot$ 	&	*	&	*	&	$\cdot$ 	&	$\cdot$ 	\\[2pt] \hline

	\end{tabular}																					
	\caption*{\\Existence (Y), nonexistence ($\cdot$), nonexistence found by computer (*), nonexistence shown in this																					
		paper (N), and open cases (?) for $CW(v, s^2)$}																					
\end{table}

\begin{table}[t]																					
	\footnotesize																					
	\centering																					
	\caption*{\Cref{table} (continued)}																					

	\begin{tabular}{|p{0.5cm}||p{0.5cm}|p{0.5cm}|p{0.5cm}|p{0.5cm}|p{0.65cm}|p{1cm}|p{1.35cm}|p{1cm}|p{1cm}|p{1cm}|}																					
		
		\hline																					
		$s$	&	1	&	2	&	3	&	4	&	5	&	6	&	7	&	8	&	9	&	10	\\[2pt] \hline
		$v$	&		&		&		&		&		&		&		&		&		&		\\[2pt]  \hhline{|=||=|=|=|=|=|=|=|=|=|=|}
		
		91	&	Y 	&	Y 	&	Y 	&	$\cdot$ 	&	$\cdot$ 	&	Y 	&	$\cdot$  \cite{LeungMa2}	&	$\cdot$ 	&	Y 	&	$\cdot$ 	\\[2pt] \hline
		92	&	Y 	&	Y 	&	$\cdot$ 	&	$\cdot$ 	&	$\cdot$ 	&	$\cdot$  \hyperref[thm184_36]{(F)}	&	$\cdot$  \hyperref[thm158_100]{(C)}	&	$\cdot$  \hyperref[thm184_36]{(F)}	&	$\cdot$  \hyperref[thm158_100]{(C)}	&	$\cdot$ 	\\[2pt] \hline
		93	&	Y 	&	$\cdot$ 	&	$\cdot$ 	&	Y 	&	Y 	&	$\cdot$ 	&	$\cdot$  \hyperref[thm158_100]{(C)}	&	$\cdot$ 	&	$\cdot$  \hyperref[thm158_100]{(C)}	&	$\cdot$ 	\\[2pt] \hline
		94	&	Y 	&	Y 	&	$\cdot$ 	&	$\cdot$ 	&	$\cdot$ 	&	$\cdot$  \hyperref[thm184_36]{(F)}	&	$\cdot$  \hyperref[thm158_100]{(C)}	&	$\cdot$  \hyperref[thm184_36]{(F)}	&	$\cdot$  \hyperref[thm184_36]{(F)}	&	$\cdot$ 	\\[2pt] \hline
		95	&	Y 	&	$\cdot$ 	&	$\cdot$ 	&	$\cdot$ 	&	$\cdot$ 	&	$\cdot$  \hyperref[thm158_100]{(C)}	&	*	&	$\cdot$  \hyperref[thm158_100]{(C)}	&	$\cdot$  \hyperref[thm158_100]{(C)}	&	$\cdot$ 	\\[2pt] \hline
		96	&	Y 	&	Y 	&	Y 	&	$\cdot$ 	&	*	&	Y 	&	*	&	$\cdot$ 	&	$\cdot$ 	&	$\cdot$ 	\\[2pt] \hline
		97	&	Y 	&	$\cdot$ 	&	$\cdot$ 	&	$\cdot$ 	&	$\cdot$ 	&	$\cdot$  \hyperref[thm184_36]{(F)}	&	$\cdot$  \hyperref[thm158_100]{(C)}	&	$\cdot$  \hyperref[thm184_36]{(F)}	&	$\cdot$  \hyperref[thm184_36]{(F)}	&	$\cdot$ 	\\[2pt] \hline
		98	&	Y 	&	Y 	&	$\cdot$ 	&	Y 	&	$\cdot$ 	&	$\cdot$ 	&	$\cdot$  \cite{LeungMa2}	&	$\cdot$ 	&	$\cdot$ 	&	$\cdot$ 	\\[2pt] \hline
		99	&	Y 	&	$\cdot$ 	&	$\cdot$ 	&	$\cdot$ 	&	Y 	&	$\cdot$  \hyperref[thm190_100]{(E)}	&	$\cdot$ 	&	$\cdot$ \cite{ArasuBayesNabavi} \hyperref[thm158_100]{(C)}	&	$\cdot$ 	&	$\cdot$ 	\\[2pt] \hline
		100	&	Y 	&	Y 	&	$\cdot$ 	&	$\cdot$ 	&	$\cdot$ 	&	$\cdot$ 	&	$\cdot$ 	&	$\cdot$ 	&	$\cdot$ 	&	$\cdot$ 	\\[2pt] \hline
		101	&	Y 	&	$\cdot$ 	&	$\cdot$ 	&	$\cdot$ 	&	$\cdot$ 	&	$\cdot$  \hyperref[thm184_36]{(F)}	&	$\cdot$  \hyperref[thm158_100]{(C)}	&	$\cdot$  \hyperref[thm184_36]{(F)}	&	$\cdot$  \hyperref[thm184_36]{(F)}	&	$\cdot$  \hyperref[thm158_100]{(C)}	\\[2pt] \hline
		102	&	Y 	&	Y 	&	$\cdot$ 	&	$\cdot$ 	&	$\cdot$ 	&	$\cdot$ 	&	$\cdot$  \hyperref[thm158_100]{(C)}	&	$\cdot$ 	&	$\cdot$  \hyperref[thm158_100]{(C)}	&	$\cdot$ 	\\[2pt] \hline
		103	&	Y 	&	$\cdot$ 	&	$\cdot$ 	&	$\cdot$ 	&	$\cdot$ 	&	$\cdot$  \hyperref[thm184_36]{(F)}	&	$\cdot$  \hyperref[thm158_100]{(C)}	&	$\cdot$  \hyperref[thm184_36]{(F)}	&	$\cdot$  \hyperref[thm184_36]{(F)}	&	$\cdot$  \hyperref[thm158_100]{(C)}	\\[2pt] \hline
		104	&	Y 	&	Y 	&	Y 	&	$\cdot$ 	&	$\cdot$ 	&	Y 	&	$\cdot$  \hyperref[thm158_100]{(C)}	&	$\cdot$  \hyperref[thm184_36]{(F)}	&	? 	&	$\cdot$ 	\\[2pt] \hline
		105	&	Y 	&	Y 	&	$\cdot$ 	&	Y 	&	$\cdot$ 	&	? 	&	$\cdot$  \cite{LeungMa2}	&	$\cdot$ 	&	N  \hyperref[cor105_81]{(B)}	&	$\cdot$ 	\\[2pt] \hline
		106	&	Y 	&	Y 	&	$\cdot$ 	&	$\cdot$ 	&	$\cdot$ 	&	$\cdot$  \hyperref[thm184_36]{(F)}	&	$\cdot$  \hyperref[thm158_100]{(C)}	&	$\cdot$  \hyperref[thm184_36]{(F)}	&	$\cdot$  \hyperref[thm184_36]{(F)}	&	$\cdot$  \hyperref[thm158_100]{(C)}	\\[2pt] \hline
		107	&	Y 	&	$\cdot$ 	&	$\cdot$ 	&	$\cdot$ 	&	$\cdot$ 	&	$\cdot$  \hyperref[thm184_36]{(F)}	&	$\cdot$  \hyperref[thm158_100]{(C)}	&	$\cdot$  \hyperref[thm184_36]{(F)}	&	$\cdot$  \hyperref[thm184_36]{(F)}	&	$\cdot$  \hyperref[thm158_100]{(C)}	\\[2pt] \hline
		108	&	Y 	&	Y 	&	$\cdot$ 	&	$\cdot$ 	&	$\cdot$ 	&	$\cdot$ 	&	$\cdot$ 	&	$\cdot$ 	&	$\cdot$ 	&	$\cdot$ 	\\[2pt] \hline
		109	&	Y 	&	$\cdot$ 	&	$\cdot$ 	&	$\cdot$ 	&	$\cdot$ 	&	$\cdot$  \hyperref[thm184_36]{(F)}	&	$\cdot$  \hyperref[thm158_100]{(C)}	&	$\cdot$  \hyperref[thm184_36]{(F)}	&	$\cdot$  \hyperref[thm184_36]{(F)}	&	$\cdot$ 	\\[2pt] \hline
		110	&	Y 	&	Y 	&	$\cdot$ 	&	$\cdot$ 	&	$\cdot$ \cite{LeungMa}	&	$\cdot$  \hyperref[thm190_100]{(E)}	&	$\cdot$  \hyperref[thm158_100]{(C)}	&	$\cdot$  \hyperref[thm190_100]{(E)}	&	? 	&	$\cdot$ \cite{ArasuMa}	\\[2pt] \hline
		111	&	Y 	&	$\cdot$ 	&	$\cdot$ 	&	$\cdot$ 	&	$\cdot$ 	&	$\cdot$  \hyperref[thm158_100]{(C)}	&	$\cdot$ 	&	$\cdot$  \hyperref[thm158_100]{(C)}	&	$\cdot$  \hyperref[thm158_100]{(C)}	&	$\cdot$  \hyperref[thm158_100]{(C)}	\\[2pt] \hline
		112	&	Y 	&	Y 	&	$\cdot$ 	&	Y 	&	$\cdot$ \cite{Yorgov}	&	? 	&	$\cdot$  \cite{LeungMa2}	&	N \hyperref[thm190_100]{(E)}	&	$\cdot$ 	&	? 	\\[2pt] \hline
		113	&	Y 	&	$\cdot$ 	&	$\cdot$ 	&	$\cdot$ 	&	$\cdot$ 	&	$\cdot$  \hyperref[thm184_36]{(F)}	&	$\cdot$  \hyperref[thm158_100]{(C)}	&	$\cdot$  \hyperref[thm184_36]{(F)}	&	$\cdot$  \hyperref[thm184_36]{(F)}	&	$\cdot$  \hyperref[thm158_100]{(C)}	\\[2pt] \hline
		114	&	Y 	&	Y 	&	$\cdot$ 	&	$\cdot$ 	&	$\cdot$ 	&	$\cdot$  \hyperref[thm190_100]{(E)}	&	Y 	&	$\cdot$  \hyperref[thm158_100]{(C)}	&	$\cdot$  \hyperref[thm158_100]{(C)}	&	$\cdot$  \hyperref[thm190_100]{(E)}	\\[2pt] \hline
		115	&	Y 	&	$\cdot$ 	&	$\cdot$ 	&	$\cdot$ 	&	$\cdot$ 	&	$\cdot$  \hyperref[thm158_100]{(C)}	&	$\cdot$  \hyperref[thm158_100]{(C)}	&	$\cdot$  \hyperref[thm158_100]{(C)}	&	$\cdot$  \hyperref[thm158_100]{(C)}	&	$\cdot$  \hyperref[thm190_100]{(E)}	\\[2pt] \hline
		116	&	Y 	&	Y 	&	$\cdot$ 	&	$\cdot$ 	&	$\cdot$ 	&	$\cdot$  \hyperref[thm184_36]{(F)}	&	? 	&	$\cdot$  \hyperref[thm184_36]{(F)}	&	$\cdot$  \hyperref[thm184_36]{(F)}	&	$\cdot$ 	\\[2pt] \hline
		117	&	Y 	&	$\cdot$ 	&	Y 	&	$\cdot$ 	&	$\cdot$ \cite{Yorgov}	&	? 	&	$\cdot$ 	&	$\cdot$  \hyperref[thm158_100]{(C)}	&	N \hyperref[thm_fielddescent]{(A)}	&	$\cdot$ 	\\[2pt] \hline
		118	&	Y 	&	Y 	&	$\cdot$ 	&	$\cdot$ 	&	$\cdot$ 	&	$\cdot$  \hyperref[thm184_36]{(F)}	&	$\cdot$  \hyperref[thm158_100]{(C)}	&	$\cdot$  \hyperref[thm184_36]{(F)}	&	$\cdot$  \hyperref[thm184_36]{(F)}	&	$\cdot$  \hyperref[thm158_100]{(C)}	\\[2pt] \hline
		119	&	Y 	&	Y 	&	$\cdot$ 	&	$\cdot$ 	&	$\cdot$ 	&	$\cdot$ 	&	$\cdot$  \cite{LeungMa2}	&	$\cdot$  \hyperref[thm190_100]{(E)}	&	$\cdot$  \hyperref[thm158_100]{(C)}	&	$\cdot$  \hyperref[thm190_100]{(E)}	\\[2pt] \hline
		120	&	Y 	&	Y 	&	Y 	&	$\cdot$ 	&	$\cdot$ \cite{LeungMa}	&	N \hyperref[thm120_36]{(I)}	&	? 	&	$\cdot$ 	&	$\cdot$ 	&	? 	\\[2pt] \hline
		
	\end{tabular}																					
	\caption*{\\Existence (Y), nonexistence ($\cdot$), nonexistence found by computer (*), nonexistence shown in this																					
		paper (N), and open cases (?) for $CW(v, s^2)$}																					
\end{table}

\begin{table}[t]																					
	\footnotesize																					
	\centering																					
	\caption*{\Cref{table} (continued)}																					
	\begin{tabular}{|p{0.5cm}||p{0.5cm}|p{0.5cm}|p{0.5cm}|p{0.5cm}|p{0.65cm}|p{1cm}|p{1.35cm}|p{1cm}|p{1cm}|p{1cm}|}																					
		\hline																					
		$s$	&	1	&	2	&	3	&	4	&	5	&	6	&	7	&	8	&	9	&	10	\\[2pt] \hline
		$v$	&		&		&		&		&		&		&		&		&		&		\\[2pt]  \hhline{|=||=|=|=|=|=|=|=|=|=|=|}
		
		121	&	Y 	&	$\cdot$ 	&	$\cdot$ 	&	$\cdot$ 	&	$\cdot$ 	&	$\cdot$ 	&	$\cdot$  \hyperref[thm158_100]{(C)}	&	$\cdot$  \hyperref[thm158_100]{(C)}	&	Y 	&	$\cdot$ 	\\[2pt] \hline
		122	&	Y 	&	Y 	&	$\cdot$ 	&	$\cdot$ 	&	$\cdot$ 	&	$\cdot$  \hyperref[thm184_36]{(F)}	&	$\cdot$  \hyperref[thm158_100]{(C)}	&	$\cdot$  \hyperref[thm184_36]{(F)}	&	$\cdot$ 	&	$\cdot$  \hyperref[thm158_100]{(C)}	\\[2pt] \hline
		123	&	Y 	&	$\cdot$ 	&	$\cdot$ 	&	$\cdot$ 	&	$\cdot$ 	&	$\cdot$ 	&	$\cdot$  \hyperref[thm158_100]{(C)}	&	$\cdot$  \hyperref[thm158_100]{(C)}	&	$\cdot$ 	&	$\cdot$  \hyperref[thm158_100]{(C)}	\\[2pt] \hline
		124	&	Y 	&	Y 	&	$\cdot$ 	&	Y 	&	Y 	&	$\cdot$ 	&	$\cdot$  \hyperref[thm158_100]{(C)}	&	Y 	&	$\cdot$  \hyperref[thm184_36]{(F)}	&	Y 	\\[2pt] \hline
		125	&	Y 	&	$\cdot$ 	&	$\cdot$ 	&	$\cdot$ 	&	$\cdot$ 	&	$\cdot$ 	&	$\cdot$ 	&	$\cdot$ 	&	$\cdot$  \hyperref[cor105_81]{(B)}	&	$\cdot$ 	\\[2pt] \hline
		126	&	Y 	&	Y 	&	$\cdot$ 	&	Y 	&	$\cdot$ 	&	$\cdot$ 	&	N \cite{LeungMa2} \hyperref[thm_fielddescent]{(A)} 	&	Y\cite{arasudillon} 	&	$\cdot$ 	&	$\cdot$ 	\\[2pt] \hline
		127	&	Y 	&	$\cdot$ 	&	$\cdot$ 	&	$\cdot$ 	&	$\cdot$ 	&	$\cdot$ 	&	$\cdot$  \hyperref[thm158_100]{(C)}	&	Y 	&	$\cdot$  \hyperref[thm184_36]{(F)}	&	$\cdot$ 	\\[2pt] \hline
		128	&	Y 	&	Y 	&	$\cdot$ 	&	$\cdot$ 	&	$\cdot$ 	&	$\cdot$ 	&	N  \hyperref[thm_fielddescent]{(A)}	&	$\cdot$\cite{SchmidtSmith}	&	$\cdot$ 	&	$\cdot$ 	\\[2pt] \hline
		129	&	Y 	&	$\cdot$ 	&	$\cdot$ 	&	$\cdot$ 	&	$\cdot$ 	&	$\cdot$  \hyperref[thm158_100]{(C)}	&	$\cdot$ 	&	$\cdot$  \hyperref[thm158_100]{(C)}	&	$\cdot$  \hyperref[thm158_100]{(C)}	&	$\cdot$  \hyperref[thm158_100]{(C)}	\\[2pt] \hline
		130	&	Y 	&	Y 	&	Y 	&	$\cdot$ 	&	$\cdot$ 	&	Y 	&	$\cdot$  \hyperref[thm158_100]{(C)}	&	$\cdot$  \hyperref[thm158_100]{(C)}	&	? 	&	$\cdot$ 	\\[2pt] \hline
		131	&	Y 	&	$\cdot$ 	&	$\cdot$ 	&	$\cdot$ 	&	$\cdot$ 	&	$\cdot$  \hyperref[thm184_36]{(F)}	&	$\cdot$  \hyperref[thm158_100]{(C)}	&	$\cdot$  \hyperref[thm184_36]{(F)}	&	$\cdot$  \hyperref[thm184_36]{(F)}	&	$\cdot$  \hyperref[thm158_100]{(C)}	\\[2pt] \hline
		132	&	Y 	&	Y 	&	$\cdot$ 	&	$\cdot$ 	&	Y 	&	$\cdot$  \hyperref[thm190_100]{(E)}	&	$\cdot$ 	&	$\cdot$ 	&	? 	&	Y 	\\[2pt] \hline
		133	&	Y 	&	Y 	&	$\cdot$ 	&	$\cdot$ 	&	$\cdot$ \cite{Yorgov}	&	$\cdot$ 	&	$\cdot$  \cite{LeungMa2}	&	$\cdot$  \hyperref[thm158_100]{(C)}	&	$\cdot$  \hyperref[thm158_100]{(C)}	&	N \hyperref[thm190_100]{(E)}	\\[2pt] \hline
		134	&	Y 	&	Y 	&	$\cdot$ 	&	$\cdot$ 	&	$\cdot$ 	&	$\cdot$  \hyperref[thm184_36]{(F)}	&	$\cdot$  \hyperref[thm158_100]{(C)}	&	$\cdot$  \hyperref[thm184_36]{(F)}	&	$\cdot$  \hyperref[thm184_36]{(F)}	&	$\cdot$  \hyperref[thm158_100]{(C)}	\\[2pt] \hline
		135	&	Y 	&	$\cdot$ 	&	$\cdot$ 	&	$\cdot$ 	&	$\cdot$ 	&	$\cdot$ 	&	$\cdot$ 	&	$\cdot$ 	&	$\cdot$  \hyperref[cor105_81]{(B)}	&	$\cdot$ 	\\[2pt] \hline
		136	&	Y 	&	Y 	&	$\cdot$ 	&	$\cdot$ 	&	$\cdot$ 	&	$\cdot$ 	&	$\cdot$  \hyperref[thm158_100]{(C)}	&	$\cdot$ 	&	$\cdot$  \hyperref[thm158_100]{(C)}	&	$\cdot$ 	\\[2pt] \hline
		137	&	Y 	&	$\cdot$ 	&	$\cdot$ 	&	$\cdot$ 	&	$\cdot$ 	&	$\cdot$  \hyperref[thm184_36]{(F)}	&	$\cdot$  \hyperref[thm158_100]{(C)}	&	$\cdot$  \hyperref[thm184_36]{(F)}	&	$\cdot$  \hyperref[thm184_36]{(F)}	&	$\cdot$  \hyperref[thm158_100]{(C)}	\\[2pt] \hline
		138	&	Y 	&	Y 	&	$\cdot$ 	&	$\cdot$ 	&	$\cdot$ 	&	N \hyperref[thm138_36]{(D)}	&	$\cdot$  \hyperref[thm158_100]{(C)}	&	$\cdot$  \hyperref[thm158_100]{(C)}	&	$\cdot$  \hyperref[thm190_100]{(E)}	&	$\cdot$  \hyperref[thm190_100]{(E)}	\\[2pt] \hline
		139	&	Y 	&	$\cdot$ 	&	$\cdot$ 	&	$\cdot$ 	&	$\cdot$ 	&	$\cdot$  \hyperref[thm184_36]{(F)}	&	$\cdot$  \hyperref[thm158_100]{(C)}	&	$\cdot$  \hyperref[thm184_36]{(F)}	&	$\cdot$  \hyperref[thm184_36]{(F)}	&	$\cdot$  \hyperref[thm158_100]{(C)}	\\[2pt] \hline
		140	&	Y 	&	Y 	&	$\cdot$ 	&	Y 	&	$\cdot$ 	&	? 	&	$\cdot$ \cite{LeungMa2}	&	? 	&	$\cdot$ 	&	$\cdot$ 	\\[2pt] \hline
		141	&	Y 	&	$\cdot$ 	&	$\cdot$ 	&	$\cdot$ 	&	$\cdot$ 	&	$\cdot$  \hyperref[thm158_100]{(C)}	&	$\cdot$  \hyperref[thm158_100]{(C)}	&	$\cdot$  \hyperref[thm158_100]{(C)}	&	$\cdot$  \hyperref[thm158_100]{(C)}	&	$\cdot$  \hyperref[thm158_100]{(C)}	\\[2pt] \hline
		142	&	Y 	&	Y 	&	$\cdot$ 	&	$\cdot$ 	&	Y 	&	$\cdot$  \hyperref[thm184_36]{(F)}	&	$\cdot$  \hyperref[thm158_100]{(C)}	&	$\cdot$  \hyperref[thm184_36]{(F)}	&	$\cdot$  \hyperref[thm184_36]{(F)}	&	Y 	\\[2pt] \hline
		143	&	Y 	&	$\cdot$ 	&	Y 	&	$\cdot$ 	&	$\cdot$ 	&	? 	&	$\cdot$  \hyperref[thm158_100]{(C)}	&	$\cdot$  \hyperref[thm158_100]{(C)}	&	? 	&	$\cdot$ 	\\[2pt] \hline
		144	&	Y 	&	Y 	&	Y 	&	$\cdot$ 	&	$\cdot$ \cite{LeungMa}	&	Y \cite{SchmidtSmith}	&	? 	&	$\cdot$ 	&	$\cdot$ 	&	$\cdot$ 	\\[2pt] \hline
		145	&	Y 	&	$\cdot$ 	&	$\cdot$ 	&	$\cdot$ 	&	$\cdot$ 	&	$\cdot$  \hyperref[thm158_100]{(C)}	&	$\cdot$ 	&	$\cdot$  \hyperref[thm158_100]{(C)}	&	$\cdot$  \hyperref[thm158_100]{(C)}	&	$\cdot$ 	\\[2pt] \hline
		146	&	Y 	&	Y 	&	$\cdot$ 	&	$\cdot$ 	&	$\cdot$ 	&	$\cdot$ 	&	$\cdot$  \hyperref[thm158_100]{(C)}	&	Y 	&	$\cdot$  \hyperref[thm184_36]{(F)}	&	$\cdot$ 	\\[2pt] \hline
		147	&	Y 	&	Y 	&	$\cdot$ 	&	Y 	&	$\cdot$ 	&	$\cdot$ 	&	$\cdot$ \cite{LeungMa2}	&	N  \hyperref[thm_fielddescent]{(A)}	&	$\cdot$  \hyperref[cor105_81]{(B)}	&	$\cdot$ 	\\[2pt] \hline
		148	&	Y 	&	Y 	&	$\cdot$ 	&	$\cdot$ 	&	$\cdot$ 	&	$\cdot$  \hyperref[thm184_36]{(F)}	&	$\cdot$ \cite{Yorgov} \hyperref[thm158_100]{(C)}	&	$\cdot$  \hyperref[thm184_36]{(F)}	&	$\cdot$  \hyperref[thm158_100]{(C)}	&	$\cdot$  \hyperref[thm158_100]{(C)}	\\[2pt] \hline
		149	&	Y 	&	$\cdot$ 	&	$\cdot$ 	&	$\cdot$ 	&	$\cdot$ 	&	$\cdot$  \hyperref[thm184_36]{(F)}	&	$\cdot$  \hyperref[thm158_100]{(C)}	&	$\cdot$  \hyperref[thm184_36]{(F)}	&	$\cdot$  \hyperref[thm184_36]{(F)}	&	$\cdot$  \hyperref[thm158_100]{(C)}	\\[2pt] \hline
		150	&	Y 	&	Y 	&	$\cdot$ 	&	$\cdot$ 	&	$\cdot$ 	&	$\cdot$ 	&	$\cdot$ 	&	$\cdot$ 	&	$\cdot$ 	&	$\cdot$ 	\\[2pt] \hline

	\end{tabular}																					
	\caption*{\\Existence (Y), nonexistence ($\cdot$), nonexistence found by computer (*), nonexistence shown in this																					
		paper (N), and open cases (?) for $CW(v, s^2)$}																					
\end{table}

\begin{table}[t]																					
	\footnotesize																					
	\centering																					
	\caption*{\Cref{table} (continued)}																					

	\begin{tabular}{|p{0.5cm}||p{0.5cm}|p{0.5cm}|p{0.5cm}|p{0.5cm}|p{0.65cm}|p{1cm}|p{1.35cm}|p{1cm}|p{1cm}|p{1cm}|}																					
		
		\hline																					
		$s$	&	1	&	2	&	3	&	4	&	5	&	6	&	7	&	8	&	9	&	10	\\[2pt] \hline
		$v$	&		&		&		&		&		&		&		&		&		&		\\[2pt]  \hhline{|=||=|=|=|=|=|=|=|=|=|=|}
		
		151	&	Y 	&	$\cdot$ 	&	$\cdot$ 	&	$\cdot$ 	&	$\cdot$ 	&	$\cdot$ 	&	$\cdot$  \hyperref[thm158_100]{(C)}	&	$\cdot$  \hyperref[thm184_36]{(F)}	&	$\cdot$  \hyperref[thm184_36]{(F)}	&	$\cdot$  \hyperref[thm158_100]{(C)}	\\[2pt] \hline
		152	&	Y 	&	Y 	&	$\cdot$ 	&	$\cdot$ 	&	$\cdot$ \cite{Yorgov}	&	$\cdot$  \hyperref[thm184_36]{(F)}	&	? 	&	$\cdot$  \hyperref[thm184_36]{(F)}	&	$\cdot$  \hyperref[thm158_100]{(C)}	&	$\cdot$  \hyperref[thm190_100]{(E)}	\\[2pt] \hline
		153	&	Y 	&	$\cdot$ 	&	$\cdot$ 	&	$\cdot$ 	&	$\cdot$ 	&	$\cdot$ 	&	$\cdot$ 	&	$\cdot$ 	&	$\cdot$  \hyperref[thm138_36]{(D)}	&	$\cdot$ 	\\[2pt] \hline
		154	&	Y 	&	Y 	&	$\cdot$ 	&	Y 	&	$\cdot$ 	&	$\cdot$ \cite{ArasuAli}	&	$\cdot$  \cite{LeungMa2}	&	$\cdot$ 	&	? 	&	N \hyperref[thm154_100]{(G)}	\\[2pt] \hline
		155	&	Y 	&	$\cdot$ 	&	$\cdot$ 	&	Y 	&	Y 	&	N \hyperref[thm155_36]{(J)}	&	$\cdot$  \hyperref[thm158_100]{(C)}	&	$\cdot$ 	&	$\cdot$  \hyperref[thm158_100]{(C)}	&	? 	\\[2pt] \hline
		156	&	Y 	&	Y 	&	Y 	&	$\cdot$ 	&	$\cdot$ \cite{LeungMa}	&	Y 	&	$\cdot$  \hyperref[thm158_100]{(C)}	&	$\cdot$ 	&	? 	&	? 	\\[2pt] \hline
		157	&	Y 	&	$\cdot$ 	&	$\cdot$ 	&	$\cdot$ 	&	$\cdot$ 	&	$\cdot$  \hyperref[thm184_36]{(F)}	&	$\cdot$  \hyperref[thm158_100]{(C)}	&	$\cdot$  \hyperref[thm184_36]{(F)}	&	$\cdot$  \hyperref[thm184_36]{(F)}	&	$\cdot$  \hyperref[thm158_100]{(C)}	\\[2pt] \hline
		158	&	Y 	&	Y 	&		&	$\cdot$ 	&	$\cdot$ 	&	$\cdot$  \hyperref[thm184_36]{(F)}	&	$\cdot$  \hyperref[thm158_100]{(C)}	&	$\cdot$  \hyperref[thm184_36]{(F)}	&	$\cdot$  \hyperref[thm184_36]{(F)}	&	N \hyperref[thm158_100]{(C)}	\\[2pt] \hline
		159	&	Y 	&	$\cdot$ 	&	$\cdot$ 	&	$\cdot$ 	&	$\cdot$ 	&	$\cdot$  \hyperref[thm158_100]{(C)}	&	$\cdot$  \hyperref[thm158_100]{(C)}	&	$\cdot$  \hyperref[thm158_100]{(C)}	&	$\cdot$  \hyperref[thm158_100]{(C)}	&	$\cdot$  \hyperref[thm158_100]{(C)}	\\[2pt] \hline
		160	&	Y 	&	Y 	&	$\cdot$ 	&	$\cdot$ 	&	$\cdot$ \cite{LeungMa}	&	$\cdot$ 	&	? 	&	$\cdot$ 	&	? 	&	N \hyperref[thm_fielddescent]{(A)}	\\[2pt] \hline
		161	&	Y 	&	Y 	&	$\cdot$ 	&	$\cdot$ 	&	$\cdot$ 	&	$\cdot$ 	&	$\cdot$ \cite{LeungMa2}	&	$\cdot$ 	&	$\cdot$  \hyperref[thm158_100]{(C)}	&	$\cdot$  \hyperref[thm190_100]{(E)}	\\[2pt] \hline
		162	&	Y 	&	Y 	&	$\cdot$ 	&	$\cdot$ 	&	$\cdot$ 	&	$\cdot$ 	&	$\cdot$ \cite{Yorgov}	&	$\cdot$ 	&	$\cdot$ 	&	$\cdot$ 	\\[2pt] \hline
		163	&	Y 	&	$\cdot$ 	&	$\cdot$ 	&	$\cdot$ 	&	$\cdot$ 	&	$\cdot$  \hyperref[thm184_36]{(F)}	&	$\cdot$  \hyperref[thm158_100]{(C)}	&	$\cdot$  \hyperref[thm184_36]{(F)}	&	$\cdot$  \hyperref[thm184_36]{(F)}	&	$\cdot$  \hyperref[thm158_100]{(C)}	\\[2pt] \hline
		164	&	Y 	&	Y 	&	$\cdot$ 	&	$\cdot$ 	&	$\cdot$ 	&	$\cdot$ 	&	$\cdot$  \hyperref[thm158_100]{(C)}	&	$\cdot$  \hyperref[thm184_36]{(F)}	&	$\cdot$ 	&	$\cdot$  \hyperref[thm158_100]{(C)}	\\[2pt] \hline
		165	&	Y 	&	$\cdot$ 	&	$\cdot$ 	&	$\cdot$ 	&	Y 	&	$\cdot$ 	&	$\cdot$ \cite{Yorgov}	&	$\cdot$ 	&	$\cdot$ 	&	? 	\\[2pt] \hline
		166	&	Y 	&	Y 	&	$\cdot$ 	&	$\cdot$ 	&	$\cdot$ 	&	$\cdot$  \hyperref[thm184_36]{(F)}	&	$\cdot$  \hyperref[thm158_100]{(C)}	&	$\cdot$  \hyperref[thm184_36]{(F)}	&	$\cdot$  \hyperref[thm184_36]{(F)}	&	$\cdot$  \hyperref[thm158_100]{(C)}	\\[2pt] \hline
		167	&	Y 	&	$\cdot$ 	&	$\cdot$ 	&	$\cdot$ 	&	$\cdot$ 	&	$\cdot$  \hyperref[thm184_36]{(F)}	&	$\cdot$  \hyperref[thm158_100]{(C)}	&	$\cdot$  \hyperref[thm184_36]{(F)}	&	$\cdot$  \hyperref[thm184_36]{(F)}	&	$\cdot$  \hyperref[thm158_100]{(C)}	\\[2pt] \hline
		168	&	Y 	&	Y 	&	Y 	&	Y 	&	$\cdot$ 	&	Y 	&	$\cdot$  \cite{LeungMa2}	&	Y 	&	$\cdot$ 	&	$\cdot$ 	\\[2pt] \hline
		169	&	Y 	&	$\cdot$ 	&	Y 	&	$\cdot$ 	&	$\cdot$ 	&	$\cdot$ 	&	$\cdot$  \hyperref[thm158_100]{(C)}	&	$\cdot$  \hyperref[thm184_36]{(F)}	&	$\cdot$ 	&	$\cdot$ 	\\[2pt] \hline
		170	&	Y 	&	Y 	&	$\cdot$ 	&	$\cdot$ 	&	$\cdot$ 	&	$\cdot$ 	&	$\cdot$  \hyperref[thm158_100]{(C)}	&	$\cdot$ 	&	$\cdot$  \hyperref[thm158_100]{(C)}	&	$\cdot$ 	\\[2pt] \hline
		171	&	Y 	&	$\cdot$ 	&	$\cdot$ 	&	$\cdot$ 	&	$\cdot$ \cite{Yorgov}	&	$\cdot$ 	&	Y 	&	$\cdot$  \hyperref[thm158_100]{(C)}	&	$\cdot$  \hyperref[thm138_36]{(D)}	&	$\cdot$ 	\\[2pt] \hline
		172	&	Y 	&	Y 	&	$\cdot$ 	&	$\cdot$ 	&	$\cdot$ 	&	$\cdot$  \hyperref[thm158_100]{(C)}	&	$\cdot$ 	&	$\cdot$  \hyperref[thm184_36]{(F)}	&	$\cdot$  \hyperref[thm184_36]{(F)}	&	$\cdot$ 	\\[2pt] \hline
		173	&	Y 	&	$\cdot$ 	&	$\cdot$ 	&	$\cdot$ 	&	$\cdot$ 	&	$\cdot$  \hyperref[thm184_36]{(F)}	&	$\cdot$  \hyperref[thm158_100]{(C)}	&	$\cdot$  \hyperref[thm184_36]{(F)}	&	$\cdot$  \hyperref[thm184_36]{(F)}	&	$\cdot$  \hyperref[thm158_100]{(C)}	\\[2pt] \hline
		174	&	Y 	&	Y 	&	$\cdot$ 	&	$\cdot$ 	&	$\cdot$ 	&	$\cdot$  \hyperref[thm138_36]{(D)}	&	Y 	&	$\cdot$  \hyperref[thm158_100]{(C)}	&	$\cdot$  \hyperref[thm158_100]{(C)}	&	$\cdot$ 	\\[2pt] \hline
		175	&	Y 	&	Y 	&	$\cdot$ 	&	$\cdot$ 	&	$\cdot$ 	&	$\cdot$ 	&	$\cdot$  \cite{LeungMa2}	&	$\cdot$  \hyperref[thm190_100]{(E)}	&	$\cdot$  \hyperref[cor105_81]{(B)}	&	$\cdot$  \hyperref[thm190_100]{(E)}	\\[2pt] \hline
		176	&	Y 	&	Y 	&	$\cdot$ 	&	$\cdot$ 	&	$\cdot$ 	&	$\cdot$ 	&	$\cdot$  \hyperref[thm158_100]{(C)}	&	$\cdot$  \hyperref[thm190_100]{(E)}	&	$\cdot$ 	&	N \hyperref[thm_fielddescent]{(A)}	\\[2pt] \hline
		177	&	Y 	&	$\cdot$ 	&	$\cdot$ 	&	$\cdot$ 	&	$\cdot$ 	&	$\cdot$  \hyperref[thm158_100]{(C)}	&	$\cdot$  \hyperref[thm158_100]{(C)}	&	$\cdot$  \hyperref[thm158_100]{(C)}	&	$\cdot$  \hyperref[thm158_100]{(C)}	&	$\cdot$  \hyperref[thm158_100]{(C)}	\\[2pt] \hline
		178	&	Y 	&	Y 	&	$\cdot$ 	&	$\cdot$ 	&	$\cdot$ 	&	$\cdot$  \hyperref[thm184_36]{(F)}	&	$\cdot$  \hyperref[thm158_100]{(C)}	&	$\cdot$  \hyperref[thm184_36]{(F)}	&	$\cdot$  \hyperref[thm184_36]{(F)}	&	$\cdot$ 	\\[2pt] \hline
		179	&	Y 	&	$\cdot$ 	&	$\cdot$ 	&	$\cdot$ 	&	$\cdot$ 	&	$\cdot$  \hyperref[thm184_36]{(F)}	&	$\cdot$  \hyperref[thm158_100]{(C)}	&	$\cdot$  \hyperref[thm184_36]{(F)}	&	$\cdot$  \hyperref[thm184_36]{(F)}	&	$\cdot$  \hyperref[thm158_100]{(C)}	\\[2pt] \hline
		180	&	Y 	&	Y 	&	$\cdot$ 	&	$\cdot$ 	&	$\cdot$ 	&	? 	&	$\cdot$ 	&	? 	&	$\cdot$ 	&	$\cdot$ 	\\[2pt] \hline

	\end{tabular}																					
	\caption*{\\Existence (Y), nonexistence ($\cdot$), nonexistence found by computer (*), nonexistence shown in this																					
		paper (N), and open cases (?) for $CW(v, s^2)$}																					
\end{table}

\begin{table}[t]																					
	\footnotesize																					
	\centering																					
	\caption*{\Cref{table} (continued)}																					

	\begin{tabular}{|p{0.5cm}||p{0.5cm}|p{0.5cm}|p{0.5cm}|p{0.5cm}|p{0.65cm}|p{1cm}|p{1.35cm}|p{1cm}|p{1cm}|p{1cm}|}																					
		
		\hline																					
		$s$	&	1	&	2	&	3	&	4	&	5	&	6	&	7	&	8	&	9	&	10	\\[2pt] \hline
		$v$	&		&		&		&		&		&		&		&		&		&		\\[2pt]  \hhline{|=||=|=|=|=|=|=|=|=|=|=|}
		
		181	&	Y 	&	$\cdot$ 	&	$\cdot$ 	&	$\cdot$ 	&	$\cdot$ 	&	$\cdot$  \hyperref[thm184_36]{(F)}	&	$\cdot$  \hyperref[thm158_100]{(C)}	&	$\cdot$  \hyperref[thm184_36]{(F)}	&	$\cdot$  \hyperref[thm184_36]{(F)}	&	$\cdot$  \hyperref[thm158_100]{(C)}	\\[2pt] \hline
		182	&	Y 	&	Y 	&	Y 	&	Y 	&	$\cdot$ 	&	Y 	&	$\cdot$  \cite{LeungMa2}	&	? 	&	Y 	&	? 	\\[2pt] \hline
		183	&	Y 	&	$\cdot$ 	&	$\cdot$ 	&	$\cdot$ 	&	$\cdot$ 	&	$\cdot$  \hyperref[thm158_100]{(C)}	&	$\cdot$  \hyperref[thm158_100]{(C)}	&	$\cdot$  \hyperref[thm158_100]{(C)}	&	$\cdot$ 	&	$\cdot$  \hyperref[thm158_100]{(C)}	\\[2pt] \hline
		184	&	Y 	&	Y 	&	$\cdot$ 	&	$\cdot$ 	&	$\cdot$ 	&	N \hyperref[thm184_36]{(F)}	&	$\cdot$  \hyperref[thm158_100]{(C)}	&	N \hyperref[thm138_36]{(D)}	&	N \hyperref[thm190_100]{(E)}	&	$\cdot$  \hyperref[thm190_100]{(E)}	\\[2pt] \hline
		185	&	Y 	&	$\cdot$ 	&	$\cdot$ 	&	$\cdot$ 	&	$\cdot$ 	&	$\cdot$  \hyperref[thm158_100]{(C)}	&	$\cdot$  \hyperref[thm158_100]{(C)}	&	$\cdot$  \hyperref[thm158_100]{(C)}	&	$\cdot$  \hyperref[thm158_100]{(C)}	&	$\cdot$  \hyperref[thm158_100]{(C)}	\\[2pt] \hline
		186	&	Y 	&	Y 	&	$\cdot$ 	&	Y 	&	Y 	&	$\cdot$ 	&	$\cdot$  \hyperref[thm158_100]{(C)}	&	Y 	&	$\cdot$  \hyperref[thm158_100]{(C)}	&	Y 	\\[2pt] \hline
		187	&	Y 	&	$\cdot$ 	&	$\cdot$ 	&	$\cdot$ 	&	$\cdot$ 	&	$\cdot$  \hyperref[thm190_100]{(E)}	&	$\cdot$  \hyperref[thm158_100]{(C)}	&	$\cdot$  \hyperref[thm158_100]{(C)}	&	$\cdot$ 	&	$\cdot$ 	\\[2pt] \hline
		188	&	Y 	&	Y 	&	$\cdot$ 	&	$\cdot$ 	&	$\cdot$ 	&	$\cdot$  \hyperref[thm184_36]{(F)}	&	$\cdot$  \hyperref[thm158_100]{(C)}	&	$\cdot$  \hyperref[thm184_36]{(F)}	&	$\cdot$  \hyperref[thm184_36]{(F)}	&	$\cdot$  \hyperref[thm158_100]{(C)}	\\[2pt] \hline
		189	&	Y 	&	Y 	&	$\cdot$ 	&	Y 	&	$\cdot$ 	&	$\cdot$ 	&	N \cite{LeungMa2}  \hyperref[thm_fielddescent]{(A)}	&	$\cdot$ 	&	$\cdot$  \hyperref[cor105_81]{(B)}	&	$\cdot$ 	\\[2pt] \hline
		190	&	Y 	&	Y 	&	$\cdot$ 	&	$\cdot$ 	&	$\cdot$ \cite{LeungMa}	&	$\cdot$  \hyperref[thm158_100]{(C)}	&	$\cdot$ \cite{Yorgov}	&	$\cdot$  \hyperref[thm158_100]{(C)}	&	$\cdot$  \hyperref[thm158_100]{(C)}	&	N \hyperref[thm190_100]{(E)}	\\[2pt] \hline
		191	&	Y 	&	$\cdot$ 	&	$\cdot$ 	&	$\cdot$ 	&	$\cdot$ 	&	$\cdot$  \hyperref[thm184_36]{(F)}	&	$\cdot$  \hyperref[thm158_100]{(C)}	&	$\cdot$  \hyperref[thm184_36]{(F)}	&	$\cdot$  \hyperref[thm184_36]{(F)}	&	$\cdot$  \hyperref[thm158_100]{(C)}	\\[2pt] \hline
		192	&	Y 	&	Y 	&	Y 	&	$\cdot$ 	&	$\cdot$ \cite{LeungMa}	&	Y \cite{SchmidtSmith}	&	? 	&	$\cdot$ 	&	$\cdot$\cite{SchmidtSmith}	&	N  \hyperref[thm_fielddescent]{(A)}	\\[2pt] \hline
		193	&	Y 	&	$\cdot$ 	&	$\cdot$ 	&	$\cdot$ 	&	$\cdot$ 	&	$\cdot$  \hyperref[thm184_36]{(F)}	&	$\cdot$  \hyperref[thm158_100]{(C)}	&	$\cdot$  \hyperref[thm184_36]{(F)}	&	$\cdot$  \hyperref[thm184_36]{(F)}	&	$\cdot$  \hyperref[thm158_100]{(C)}	\\[2pt] \hline
		194	&	Y 	&	Y 	&	$\cdot$ 	&	$\cdot$ 	&	$\cdot$ 	&	$\cdot$  \hyperref[thm184_36]{(F)}	&	$\cdot$  \hyperref[thm158_100]{(C)}	&	$\cdot$  \hyperref[thm184_36]{(F)}	&	$\cdot$  \hyperref[thm184_36]{(F)}	&	$\cdot$  \hyperref[thm158_100]{(C)}	\\[2pt] \hline
		195	&	Y 	&	$\cdot$ 	&	Y 	&	$\cdot$ 	&	$\cdot$ 	&	? 	&	$\cdot$  \hyperref[thm158_100]{(C)}	&	$\cdot$  \hyperref[thm158_100]{(C)}	&	? 	&	? 	\\[2pt] \hline
		196	&	Y 	&	Y 	&	$\cdot$ 	&	Y 	&	$\cdot$ 	&	$\cdot$ 	&	$\cdot$  \cite{LeungMa2}	&	? 	&	$\cdot$ 	&	$\cdot$ 	\\[2pt] \hline
		197	&	Y 	&	$\cdot$ 	&	$\cdot$ 	&	$\cdot$ 	&	$\cdot$ 	&	$\cdot$  \hyperref[thm184_36]{(F)}	&	$\cdot$  \hyperref[thm158_100]{(C)}	&	$\cdot$  \hyperref[thm184_36]{(F)}	&	$\cdot$  \hyperref[thm184_36]{(F)}	&	$\cdot$  \hyperref[thm158_100]{(C)}	\\[2pt] \hline
		198	&	Y 	&	Y 	&	$\cdot$ 	&	$\cdot$ 	&	Y 	&	$\cdot$  \hyperref[thm190_100]{(E)}	&	$\cdot$ \cite{Yorgov}	&	$\cdot$ 	&	? 	&	Y\cite{arasudillon} 	\\[2pt] \hline
		199	&	Y 	&	$\cdot$ 	&	$\cdot$ 	&	$\cdot$ 	&	$\cdot$ 	&	$\cdot$  \hyperref[thm184_36]{(F)}	&	$\cdot$  \hyperref[thm158_100]{(C)}	&	$\cdot$  \hyperref[thm184_36]{(F)}	&	$\cdot$  \hyperref[thm184_36]{(F)}	&	$\cdot$  \hyperref[thm158_100]{(C)}	\\[2pt] \hline
		200	&	Y 	&	Y 	&	$\cdot$ 	&	$\cdot$ 	&	$\cdot$ 	&	$\cdot$ 	&	$\cdot$ 	&	$\cdot$ 	&	$\cdot$ 	&	$\cdot$ 	\\[2pt] \hline

	\end{tabular}																					
	\caption*{\\Existence (Y), nonexistence ($\cdot$), nonexistence found by computer (*), nonexistence shown in this																					
		paper (N), and open cases (?) for $CW(v, s^2)$}																					
\end{table}																	

In addition, we solved some open cases in the table of group invariant weighing matrices of \cite{arasu2013group}, where the group is abelian but non-cyclic. 
Table \ref{table4} summarizes our results. 
Note that an abelian group $G = C_{v_1} \times C_{v_2} \times \ldots \times C_{v_r}$ is represented by $[v_1, v_2, \ldots, v_r]$ in the table.

\begin{table}[!ht]												
	\footnotesize											
	\centering											
	\caption{\Cref{table4}: Non-existence of $G$-invariant Weighing Matrices $W(|G|, n)$}			\label{table4}								
	\begin{tabular}{|c||c|c|}											
		
		\hline										
		$G$		&	$n$	&	Non-existence Proof	\\[2pt] \hhline{|=||=|=|}							
		$	[ 2, 2, 11 ]	$	& $	9	$ &	\Cref{cor:abelian1}		\\[2pt] \hline			
		$	[ 2, 2, 2, 11 ]	$	& $	9	$ &	\Cref{cor:abelian1}		\\[2pt] \hline			
		$	[ 2, 2, 23 ]	$	& $	9	$ &	\Cref{cor:abelian1}		\\[2pt] \hline			
		$	[ 2, 4, 11 ]	$	& $	9	$ &	\Cref{cor:abelian1}		\\[2pt] \hline			
		$	[ 2, 2, 19 ]	$	& $	25	$ &	\Cref{cor:abelian1}		\\[2pt] \hline			
		$	[ 2, 2, 11 ]	$	& $	36	$ &	\Cref{cor:abelian4}		\\[2pt] \hline			
		$	[ 2, 2, 23 ]	$	& $	36	$ &	\Cref{cor:abelian1}		\\[2pt] \hline			
		$	[ 2, 4, 11 ]	$	& $	36	$ &	\Cref{cor:abelian4}		\\[2pt] \hline			
		$	[ 3, 27 ]	$	& $	49	$ &	\Cref{thm_fielddescent}		\\[2pt] \hline			
		$	[ 2, 32 ]	$	& $	64	$ &	\Cref{thm_fielddescent}		\\[2pt] \hline			
		$	[ 2, 2, 23 ]	$	& $	81	$ &	\Cref{cor:abelian1}		\\[2pt] \hline			
		
	\end{tabular}											
\end{table}

\end{document}